\documentclass[journal]{IEEEtran}
\pdfoutput=1
\ifCLASSINFOpdf
\else
\fi

\usepackage{amsmath,amssymb}
\usepackage{amsthm}
\usepackage{graphicx}
\usepackage{subcaption}
\usepackage{cleveref}
\usepackage{float}
\usepackage{enumerate}
\usepackage{tikz}
\usepackage{booktabs}
\usepackage{bm}
\usepackage{color}
\usepackage[linesnumbered,boxed,ruled,commentsnumbered]{algorithm2e}
\usepackage{bbm}
\usepackage{epstopdf}
\usepackage{cite}
\usepackage{textcomp}
\allowdisplaybreaks

\usepackage{algorithmic}

\begin{document}
	\bibliographystyle{ieeetr}
	
	\newtheorem{Theorem}{\textbf{Theorem}}
	\newtheorem{Lemma}{\textbf{Lemma}}
	\newtheorem{Definition}{\textbf{Definition}}
	\newtheorem{Remark}{\textbf{Remark}}
	\newtheorem{Assumption}{\textbf{Assumption}}
	\newtheorem{Corollary}{\textbf{Corollary}}
	\newtheorem{Proof}{Proof}
	\newtheorem{Example}{\textbf{Example}}
	
	\makeatletter
	\renewcommand*\env@matrix[1][\arraystretch]{%
		\edef\arraystretch{#1}%
		\hskip -\arraycolsep
		\let\@ifnextchar\new@ifnextchar
		\array{*\c@MaxMatrixCols c}}
	\makeatother

	\title{Distributed Adaptive Gradient Algorithm with Gradient Tracking for Stochastic Non-Convex Optimization}
	
	\author{Dongyu~Han,
		Kun~Liu,~\IEEEmembership{Senior~Member,~IEEE,}
		Yeming~Lin,
		Yuanqing~Xia,~\IEEEmembership{Fellow,~IEEE}
		\thanks{This work was supported in part by the National Natural Science Foundation of China under Grants 62273041 and 61873034, in part by the Joint Open Foundation of the State Key Laboratory of Synthetical Automation for Process Industries under Grant 2021-KF-21-05, in part by the Graduate Research and Innovation Training Foundation of Beijing Institute of Technology under Grant 2023YCXY034. \emph{(Corresponding author: Kun Liu.)}}
		\thanks{Dongyu~Han, Kun~Liu, Yeming~Lin and Yuanqing~Xia are with the School of Automation,
			Beijing Institute of Technology, Beijing 100081, China (e-mails: handongyu@bit.edu.cn; kunliubit@bit.edu.cn;  yeminglin@bit.edu.cn; xia\_yuanqing@bit.edu.cn).}}
	
	\markboth{}%
	{Shell \MakeLowercase{\textit{et al.}}: Bare Demo of IEEEtran.cls for IEEE Journals}
	
	\maketitle
	
	\begin{abstract}
		This paper considers a distributed stochastic non-convex optimization problem, where the nodes in a network cooperatively minimize a sum of $L$-smooth local cost functions with sparse gradients. 
		By adaptively adjusting the stepsizes according to the historical (possibly sparse) gradients, a distributed adaptive gradient algorithm is proposed, in which a gradient tracking estimator is used to handle the heterogeneity between different local cost functions. 
		We establish an upper bound on the optimality gap, which indicates that our proposed algorithm can reach a first-order stationary solution dependent on the upper bound on the variance of the stochastic gradients. 
		Finally, numerical examples are presented to illustrate the effectiveness of the algorithm.
	\end{abstract}
	
	
	\IEEEpeerreviewmaketitle

	\section{Introduction}
\IEEEPARstart{W}{ITH} the rapid development of big data, distributed optimization has raised interest in the fields of signal processing, machine learning and robot networks~\cite{patterson2014distributed,deng2009imagenet,jiang2022fully} due to its advantages in high computation and communication efficiency as well as the robustness to network uncertainty~\cite{sun2020improving}. 
In distributed optimization, the computation nodes usually need to communicate with each other to minimize a finite-sum cost function, where each node only has access to partial knowledge of the entire network. 

According to the convexity of local cost functions and constraints, most studies on distributed optimization can be classified into two categories: distributed convex optimization and distributed non-convex optimization. 
Distributed convex optimization has been widely investigated, and various algorithms have been developed over the past decade, e.g., distributed gradient descent~\cite{nedic2009distributed,nedic2014distributed,zhang2019asyspa,lu2020privacy}, distributed primal-dual gradient~\cite{zhu2011distributed,li2020distributed}, distributed dual averaging~\cite{duchi2011dual,han2021privacy}, distributed mirror descent~\cite{yuan2020distributed,yi2020distributedmirror} and distributed proximal gradient algorithms~\cite{shi2015proximal,dixit2020online}.
On the other hand, non-convex optimization problems also widely exist in practical scenarios. 
For example, in deep neural networks, the cost function is usually non-convex due to the interaction of multiple hidden layers with nonlinear activation functions~\cite{cui2020multicomposite}, and in robotic networks, the physical constraints of robots sometimes are highly non-convex\cite{breitenmoser2010voronoi}. 
Compared to convex optimization problems, the algorithm design and analysis of non-convex optimization problems are relatively complex because of the absence of good properties of convexity~\cite{jiang2021distributed}. 
Therefore, it is also of great necessity to investigate algorithms for distributed non-convex optimization problems.	
The distributed gradient descent algorithm for non-convex optimization problem was investigated in~\cite{tatarenko2017non}, where the cost functions are assumed to be $L$-smooth. In~\cite{assran2019stochastic}, a Distributed Stochastic Gradient Descent (DSGD) algorithm, in which the iterative solution can achieve a local minimum called $\epsilon$-accurate stationary solution, was further developed for stochastic non-convex optimization problems, while the heterogeneity between local cost functions is not considered.
It was investigated in~\cite{koloskova2020unified} and~\cite{lian18Asynchronous} that the heterogeneity over the network could affect the stationarity performance of the DSGD algorithm for non-convex optimization problems through an additional bias term on the optimality gap, and the iterative solution may suffer from a consensus error~\cite{alghunaim2021unified}. 

To deal with the heterogeneity, a distributed Gradient Tracking (GT) algorithm was proposed in~\cite{tziotis2020second} for non-convex optimization problems over balanced networks. In distributed GT algorithm, an additional GT vector is introduced to estimate the global gradient of the whole cost function, which helps to find a stationary solution under the impact of heterogeneity. 
A stochastic variant of distributed GT algorithm was developed in~\cite{zhang2019decentralized} for the distributed stochastic empirical risk minimization problem with non-convex and $L$-smooth cost functions.
Furthermore, a distributed derivative-free GT algorithm was proposed in~\cite{tang2020distributed}, where the zero-order stochastic GT estimator is used in the iterations. 
In~\cite{kungurtsev2021decentralized}, the distributed stochastic GT algorithm was also extended to non-convex optimization problems over unbalanced networks. 
More recently, several variance reduction techniques were employed in distributed GT algorithm to reduce the variances of the stochastic gradient estimators~\cite{xin2020variance,jiang2021distributed,sun2020improving}. 
Note that the above algorithms involve the use of pre-designed stepsizes in the iterations.

On the other hand, the data in real-world applications is usually complicated, which may result in the sparse gradients. 
When dealing with such sparsity, the aforementioned algorithms with pre-designed stepsizes may suffer from slow convergence with bad-chosen hyper-parameters, and show poor performance since the gradient is scaled uniformly in all dimensions~\cite{chen2021cada,zhang2021distributed}. 
To overcome this problem, adaptive gradient algorithms~\cite{duchi2011adaptive,zeiler2012adadelta,reddi2019convergence,kingma2014adam,chen2018convergence} have received increasing attention since less tweaking of hyper-parameters are required to achieve satisfactory performance even if the gradients are sparse.

Moreover, a distributed adaptive gradient algorithm based on momentum gradient decent was presented in~\cite{shen2020distributed} under the convex setting. In~\cite{shen2020distributed}, each dimension of the gradient is rescaled by adjusting the stepsizes based on the historical gradients, which leads to good performance on problems with sparse gradients. 
A distributed adaptive gradient algorithm with bounded stepsizes was further studied in~\cite{zhang2021distributed} to improve the generalization capacity. 
By introducing a GT estimator, a novel distributed adaptive algorithm was developed in the notable work~\cite{carnevale2020distributed}, which is proved to achieve a linear convergence rate under the strongly-convex setting. 

Motivated by the above discussion, this paper investigates a distributed adaptive gradient algorithm for addressing distributed stochastic non-convex optimization problems with $L$-smooth cost functions. 
The main contributions are as follows:

\begin{itemize}
	\item[(a)] We propose a GT-based distributed adaptive gradient algorithm, in which each node performs a momentum gradient descent based on the adaptive stepsizes to find a stationary solution. 
	The adaptive stepsizes are generated according to the historical gradients, enabling the algorithm to automatically coordinate the stepsizes among dimensions when the gradients are sparse. 
	Inspired by\cite{carnevale2020distributed}, we utilize a GT estimator to aggregate the gradients over the network. 
	Moreover, an clipping operator is used to mitigate the negative effects of extreme adaptive stepsizes.		
	\item[(b)] We provide a rigorous stationarity analysis for our proposed algorithm under the non-convex setting. 
	We find that the GT estimator plays a crucial role in handling the heterogeneity of different local cost functions, since it can not only aggregate the directions of momentum gradient descent but also mitigate the disagreement of adaptive stepsizes between different nodes, as shown in~Lemmas~\ref{lm_mcon} and~\ref{lm_v_con}, respectively. 
	These characteristics enable our algorithm to find a stationarity solution of the distributed stochastic non-convex optimization problem.
	\item[(c)]  It is shown that the upper bound on the optimality gap is of the order $O(1/T + \sigma^2)$ with~$T$ the iteration number and~$\sigma^2$ an upper bound on the variance of the stochastic gradients, which aligns with the one observed in centralized adaptive gradient algorithm~\cite{zaheer2018adaptive} for stochastic non-convex optimization problems. 		
\end{itemize}

\textbf{Notations}: Let $\mathbb{R}$ and $\mathbb{R}^+$ denote the set of real number and positive real number, respectively.	
Denote $\mathbb{R}^n$ as the $n$-dimensional real column vector space, and let $\mathbb{R}^{m \times n}$ be the set of $m \!\times \! n$-dimensional real matrix. 
Denote the inner product of two real vectors $x,y \in\mathbb{R}^d$ by $\langle x,y \rangle$. 
The notation~$[x]_{i}$ stands for the $i$-th entry of vector $x$. 
For a matrix $A$, we use $A'$ to denote its transpose and use $[A]_{ij}$ to denote its $i,j$-th entry. 
The notations $\|A\|$ and $\|A\|_F$ denote the spectral norm and the Frobenius norm of $A$, respectively.
Denote the spectral radius of matrix $A\in \mathbb{R}^{n\times n}$ as $\rho(A)$. 
We use ${\rm diag}\{ a_1, \dots, a_n\}$ to denote the diagonal matrix that consists of the scalars $a_1, \dots, a_n$, and the block diagonal matrix ${\rm blk~ diag}\{ A_1, \dots, A_n\}$ is defined in a similar way with $A_1,\dots,A_n$ some real matrices. 
The notations $\max\{\cdot\}$ and $\min\{\cdot\}$ denote the maximum and the minimum element in~$\{\cdot\}$, respectively. 
The column vectors of all ones and zeros with size $n$ are denoted by $\bm{1}_n$ and $\bm{0}_n$, respectively. 
The~$n \!\times\! n$ dimensional identity matrix is denoted by $I_n$. 	
The Hardmard product and the Kronecker product is represented as~`$\odot$' and~`$\otimes$', respectively. 
Let $\mathbb{E}[\cdot]$ denote the expectation of random variables. 
We use $O(\cdot)$ to describe the limiting behavior of a function, e.g., for functions $f: \mathbb{R}^+\rightarrow \mathbb{R}$ and $g: \mathbb{R}^+ \rightarrow \mathbb{R}^+$, we say $f(t) \leq {O}(g(t))$ if there exist positive real numbers $M$ and $t_0$ such that $|f(t)| \leq M g(t) $ for all~$t \geq t_0$.

\section{Problem formulation}
Consider a non-convex optimization problem over a distributed network with $n$ nodes. The nodes can communicate over an undirected network $\mathcal{G} = (\mathcal{V}, \mathcal{E},A)$, where~$\mathcal{V} = \{1,2,\dots,n\}$ and $\mathcal{E}$ are the node and the edge sets, respectively, and $A = [A_{ij}] \in \mathbb{R}^{n\times n}$ represents the weighted adjacency matrix. The nodes in $\mathcal{V}$ aim to solve the following consensus-based distributed optimization problem
\begin{equation} \label{problem}
	\begin{split}
		\min_{\{x_i\}_{i=1}^n} \;\; & \frac{1}{n}\sum_{i=1}^{n} f_i(x_i) \\
		\text{s.t.} \;\; &x_i = x_j, \;\; \forall i,j \in \mathcal{V} 
	\end{split}
\end{equation}
with 
\begin{equation}\label{problem2}
	f_i(x_i) \triangleq \mathbb {E}_{\xi_i} [F_i (x_i,\xi_i)],
\end{equation}
where  $x_i\in \mathbb{R}^d$ is the decision variable of node $i$, $f_i: \mathbb{R}^d \rightarrow \mathbb{R}$ is the local cost function of node $i$ that is dependent on random variable $\xi_i$ and private (possibly non-convex) function~$F_i$. 
Node $i$ can evaluate the stochastic (possibly sparse) gradient~$ \nabla F_i (x_i,\xi_i)$ at the point~$x$ by randomly sampling~$\xi_i$ from a local distribution. 
Note that the functions $F_i$, for all~$i\in\mathcal{V}$, are allowed to be different among the nodes, and the random variables~$\xi_i$, for all~$i\in\mathcal{V}$, may be sampled from different distributions, which leads to the heterogeneity of local cost functions. 
We also denote $f(x) = \frac{1}{n} \sum_{i=1}^{n} f_i(x)$ as the average cost function over the network.

Then, we make the some assumptions on the above problem.

\begin{Assumption}\label{ass_smooth}
	For $i\in \mathcal{V}$, the cost function $f_i$ is differentiable and $L$-smooth for some positive scalar~$L$, i.e., $\|\nabla f_i(x) - \nabla f_i(y)\| \leq L \|x-y\|$, which is equivalent to $f_i(y)  \leq f_i(x) - \langle \nabla f_i(x) , x - y\rangle + \frac{L}{2} \|x - y\|^2$, holds for any $x,y\in \mathbb{R}^d$.	
\end{Assumption}

\begin{Assumption} \label{assumption_sto_grad} The stochastic gradient $\nabla F_i(x_i,\xi_i)$, $i\in \mathcal{V}$, satisfies the following conditions:
	\begin{itemize}
		\item[(a)] $\nabla F_i(x_i,\xi_i)$ is an unbiased estimate of the true gradient, i.e., $\mathbb {E}_{\xi_i} [\nabla F_i(x_i,\xi_i)] = \nabla f_i(x_i)$.
		\item[(b)] There exists a scalar $G>0$ such that $\|\nabla F_i(x_i,\xi_i)\| \leq G$ almost surely holds for any $x_i\in\mathbb{R}^d$ and $\xi_i$.	
		\item[(c)] The variance of the stochastic gradient is bounded, i.e., there exists a scalar $\sigma>0$ such that $\mathbb{E}_{\xi_i} \|\nabla F_i(x_i,\xi_i) - \nabla f_i(x_i)\|^2\leq \sigma^2$.
	\end{itemize}
	
\end{Assumption}
\begin{Remark}
	Assumption~\ref{assumption_sto_grad}(b) is commonly used in the works on adaptive gradient algorithms~\cite{ward2019adagrad,kingma2014adam,reddi2019convergence,chen2018convergence,zhang2021distributed, pmlr-v89-li19c} for stochastic non-convex optimization problems. 
	This assumption implies that the gradient of the cost function is Lipschitz bounded, i.e., $\|\nabla f_i(x_i)\| \leq G$, which can be satisfied by a wide range of cost functions, e.g., the Huber function and Geman-McClure function in robust optimization~\cite{yang2020graduated} as well as the negative log-likelihood function for logistic regression~\cite{kingma2014adam}. 
	On the other hand, Assumption~\ref{assumption_sto_grad}(b) also implies that the stochastic error between the stochastic gradient and the true gradient, i.e., $\|\nabla F_i(x_i,\xi_i) - \nabla f_i(x_i)\|$, is bounded. 
	This type of bounded stochastic error or noises, including the truncated Gaussian noise and the bounded uniform noise, commonly arises in various applications. 	
	For instance, in engineering applications, sensors often have limited measuring ranges, which leads to bounded noises.		 
	Moreover, in signal processing and data analysis tasks, extremely large noisy data is commonly considered as anomalous signal during the pre-processing, and then, is excluded from further computation. 
	As a result, these scenarios naturally lead to the almost surely bounded gradient in Assumption~\ref{assumption_sto_grad}(b).
\end{Remark}

Since node $i$ only has information about $F_i$ and $\xi_i$, all the nodes need to exchange information over the network in order to solve problem~(\ref{problem}). 
The node $j$ can receive information from node~$i$ if edge $(i,j)\in \mathcal{E}$. Denote $\mathcal{N}_i = \{j|(i,j)\in \mathcal{E}\}$ as the neighbor set of node $i$. Then, a basic assumption on the graph is given as follows:
\begin{Assumption} \label{ass_graph} The graph $\mathcal{G}$ is connected. Moreover, the weighted adjacency matrix $A$ is doubly stochastic and satisfies $\rho \big(A - \frac{\bm{1}_n \bm{1}_n'}{n}\big) < 1$.
\end{Assumption}
\begin{Remark}
	Assumption~\ref{ass_graph} can be satisfied in connected undirected networks through properly designed weight setting protocols, such as Metropolis weight protocol~\cite{nedic2017achieving}.
\end{Remark}

In our considered distributed stochastic non-convex optimization problem, since each node maintains a local solution~$x_i$, we employ the following metric:
\begin{equation}\label{gap}
	\mathbb {E} \left[ \|\nabla f(\bar x) \|^2 + \frac{1}{n} \sum_{i=1}^{n} \|x_{i} - \bar{x} \|^2 \right]
\end{equation}
with $\bar{x} = \frac{1}{n} \sum_{i=1}^{n} x_i$, to evaluate the first-order stationarity of the solutions $\{x_i\}_{i=1}^n$ of all the $n$ nodes over the network.
The metric~(\ref{gap}), also known as optimality gap~\cite{sun2020improving}, captures both the gradient norm and the consensus error of the solutions, providing a comprehensive evaluation of the stationarity performance.

\section{Algorithm Design}
In this section, we present our proposed GT-based distributed adaptive gradient algorithm, as shown in Algorithm~\ref{alg}. 

At $t$-th iteration, node~$i$ maintains a \emph{local estimate}~$x_{t,i}\in \mathbb{R}^d$ of the solution of the distributed optimization problem~(\ref{problem}) and can evaluate the stochastic gradient~$g_{t,i} \triangleq \nabla F_i(x_{t,i},\xi_i)$. 
The vector~$m_{t,i}\in \mathbb{R}^d$ denotes the \emph{momentum gradient}. 
The vector~$s_{t,i}\in \mathbb{R}^d$ represents the \emph{GT estimator}, which is designed to track the gradient~$\nabla f(x)$ of the global cost function. 
The notation~$v_{t,i} \in \mathbb{R}^d$ represents the \emph{adaptive vector} to rescale the stepsizes, while the vector~$\hat{v}_{t,i} \in \mathbb{R}^d$ and the matrix~$V_{t,i} \in \mathbb{R}^{d\times d}$ are auxiliary adaptive variables. 
In addition, the scalars~$\beta_1 \in (0,1)$ and $\beta_2\in (0,1)$ denote the exponential decay rates of $m_{t,i}$ and $v_{t,i}$, respectively, and~$\alpha>0$ is the stepsize. 

\begin{algorithm}
	\caption{GT-based distributed adaptive gradient algorithm}  
	\label{alg}  
	\begin{algorithmic}[1]
		\STATE {\bf Initialization:} Parameters $\beta_1\in (0,1),\beta_2\in (0,1)$, initial stepsize $\alpha > 0$, $v_{\rm max} \geq 1 \geq  v_{\rm min}>0$, iteration number~$T$, and the initial states
		$\forall x_{1,i} \in \mathbb{R}^d$, $s_{1,i} = g_{1,i}$, $m_{1,i} = \bm{0}_d$ and $v_{1,i} = s_{1,i} \odot s_{1,i}$. 
		\FOR {$t=1,\dots,T$, node $i\in \mathcal{V}$}
		\STATE Communicate the \emph{local estimate} $x_{t-1,i}$ and the \emph{GT estimator}~$s_{t-1,i}$ with its neighbors.\\
		\STATE Update
		\begin{align}
			m_{t+1,i} &=  \beta_1 m_{t,i} + (1-\beta_1) s_{t,i}, \label{m_update}\\
			\hat{v}_{t+1,i} &=  \beta_2 v_{t,i} + (1 - \beta_2) s_{t,i} \odot s_{t,i}, \label{v_update}\\
			{v}_{t+1,i} &=  {\rm Clip} (\hat{v}_{t+1,i}, v_{\rm min}, v_{\rm max}), \label{v3_update}\\
			V_{t+1,i} &=  {\rm diag} \{[v_{t+1,i}]_1, \dots, [v_{t+1,i}]_d\}, \label{v2_update}\\
			x_{t+1,i} &=  \sum_{j=1}^{n} A_{ij} x_{t,j} - \alpha V_{t+1,i}^{-1/2} m_{t+1,i}, \label{x_update}\\
			s_{t+1,i} &=  \sum_{j=1}^{n} A_{ij} s_{t,j} + g_{t+1,i} - g_{t,i}, \label{s_update}
		\end{align}
		where the stochastic gradient $g_{t,i} = \nabla F_i(x_{t,i},\xi_i)$.
		\ENDFOR
		\STATE {\bf Output}: The \emph{local estimate} $x_{T+1,i}$
	\end{algorithmic}
\end{algorithm}

As shown in (\ref{m_update})-(\ref{v2_update}), at $t$-th iteration, node $i$ first performs an exponential moving averaging to update the \emph{momentum gradient} $m_{t,i}$ as well as the \emph{adaptive vector} $v_{t,i}$ based on the \emph{GT estimator} $s_{t-1,i}$.  
Then, node $i$ gathers information from its neighbors $j\in \mathcal{N}_i$ and updates the \emph{local estimate} $x_{t,i}$ as in~(\ref{x_update}) based on a distributed momentum gradient descent. 
In~(\ref{x_update}) the descent direction is determined by the \emph{momentum gradient}~$m_{t,i}$, and the stepsizes in different dimensions are rescaled by matrix $\alpha V_{t,i}^{-1/2}$ to handle sparse gradients. 
Specifically, the stepsize is enlarged adaptively when it is in the dimension with a small gradient, and vice versa. 
To avoid extreme values of the adaptive stepsizes, an element-wise clipping operation 
\begin{equation}
	{\rm Clip} (v, v_{\rm min}, v_{\rm max}) \triangleq \max \{\min\{v,v_{\rm max}\},v_{\rm min} \}
\end{equation}
with $v_{\rm max}$ and $v_{\rm min}$ the upper and lower bound of $v_{t,i}$, respectively, is equipped to bound the value of $\hat{v}_{t,i}$. 
In addition, as in (\ref{s_update}), the update of the \emph{GT estimator}~$s_{t,i}$ involves an accumulation of the gradient innovations, i.e., $g_{t,i}-g_{t-1,i}$, as well as a weighted average consensus operation to track the gradient $\nabla f(x)$ of the global cost function. 	
Finally, node~$i$ outputs the \emph{local estimate} $x_{T+1,i}$ of the solution. 

\begin{Remark}
	Compared to the GT-based distributed adaptive gradient algorithm for a strongly-convex problem in~\cite{carnevale2020distributed}, our work aims to investigate the stationarity performance of the proposed Algorithm~\ref{alg} for a more challenging distributed stochastic non-convex optimization problem. 
	A notable characteristic of Algorithm~\ref{alg} is the utilization of a clipping operator as in~(\ref{v3_update}), which distinguishes our algorithm from the algorithm in~\cite{carnevale2020distributed} with an additive bias term. 
	Specifically, in~\cite{carnevale2020distributed}, an easily implemented additive bias term $\epsilon$ is employed on each~$v_{t,i}$ to ensure that the  adaptive vector $v_{t,i} + \epsilon$ is consistently greater than or equal to the constant~$\epsilon$. 
	In contrast, our clipping operator provides a more direct approach to mitigate the negative impact of extreme stepsizes, since the adaptive vector will be clipped only if $v_{t,i}$ is smaller or larger than the threshold $v_{\min}$ or $v_{\max}$, respectively. 
	Furthermore, we provide a rigorous stationarity analysis for our Algorithm~\ref{alg}, which establishes an explicit upper bound on the optimality gap, as demonstrated in  Corollary~\ref{Col_1} below. 
	It is important to note that our analysis approach can also be extended and applied to the algorithm in~\cite{carnevale2020distributed} by adjusting the coefficients~$v_{\min}, v_{\max}$ in our analysis according to the value of additive bias term~$\epsilon$ in~\cite{carnevale2020distributed}.  
\end{Remark}
	
	\section{Convergence Analysis}
	In this section, we provide the convergence analysis of our proposed algorithm for non-convex optimization problem in terms of the \emph{optimality gap} given by~(\ref{gap}). 
	
	We first introduce several auxiliary vectors
	\begin{equation}\label{z_def}
		z_{t,i} \triangleq 
		\begin{cases}
			x_{t,i} , &t = 1, \\
			\frac{1}{1-\beta_1} x_{t,i} - \frac{\beta_1}{1-\beta_1} x_{t-1,i}, &t \geq 2.
		\end{cases}
	\end{equation}	
	
	The notation
	\begin{equation}\label{bold_def}
		\bm{\zeta}_t \triangleq [\zeta_{t,1}',\dots,\zeta_{t,n}']' \in \mathbb{R}^{nd}
	\end{equation}
	aggregates all the vectors $\zeta_{t,i}$ of the nodes $i \in \{1, \dots, n\}$, while the notations
	\begin{equation}\label{bar_def}
		\begin{split}
			&\bar{\zeta}_t \triangleq \frac{1}{n} \sum_{i=1}^{n} \zeta_{t,i} \in \mathbb{R}^{d}, {\rm~and~} \tilde{\bm{\zeta}}_t \triangleq \bm{1}_n \otimes \bar{\zeta}_t \in \mathbb{R}^{nd}
		\end{split}
	\end{equation}
	represent the average vectors, where $\zeta$ can be any variables such that $\zeta\in \{x,s,z,m,v\}$. 
	Moreover, define~$V_t \triangleq {\rm blk \; diag} \{V_{t,1},\dots,V_{t,n}\} \in \mathbb{R}^{nd\times nd}$ as the aggregation form of the adaptive matrix $V_{t,i}$, and $\bar V_t \triangleq \frac{1}{n}\sum_{i=1}^{n} V_{t,i}$. 
	Let~$\eta_{t,i} =  \nabla F_i(x_{t,i},\xi_i) - \nabla f_i(x_{t,i})$ represents the error between the stochastic gradient $\nabla F_i(x_{t,i},\xi_i)$ and the true gradient~$\nabla f_i(x_{t,i})$.
	
	With auxiliary vector $\bar{z}_t$ defined in~(\ref{bar_def}), in the following lemma we establish an upper bound on the evolution of $f(\bar{z}_{t})$ based on the $L$-smooth property of cost function $f_i$.
	\begin{Lemma}\label{lm_smooth}
		Under Assumptions~\ref{ass_smooth} and \ref{assumption_sto_grad}(b), we have the following result:
		\begin{equation}\label{formula_lm_smooth}
			\begin{split}
				&f(\bar{z}_{t+1}) \!\leq\! f(\bar{z}_{t}) \!-\! \alpha \Big(\! v_{\rm max} ^{\!-\!1/2} \!-\! \frac{\alpha v_{\rm min}^{\!- \!1/2}}{2} \!-\! \alpha v_{\rm min}^{-1} (L\!+\! 1) \Big) \! \|\nabla f(\bar{x}_t)\|^2 \\
				& \quad
				+ M_1 \| {\bm{x}}_t - \tilde{\bm{x}}_t \|^2  + M_2 \| \bm{s}_{t} - \tilde{\bm{s}}_t \|^2 + M_3 \|  \bm{m}_t - \tilde{\bm{m}}_t \|^2 \\
				& \quad+ M_4 \| \tilde{\bm{m}}_t \|^2 + M_5 \| \bm{v}_{t} - \tilde{\bm{v}}_t\|^2 + M_6 \sum_{i=1}^{n} \| \eta_{t,i}\|^2
			\end{split}
		\end{equation}
		with
		\begin{align}
			M_1 & = \frac{4 v_{\rm min}^{-1} L}{n}\Big( v_{\rm max}^{1/2} + 2 \alpha^2 (L+1) \Big)  , \nonumber\\
			M_2 & = \frac{18 v_{\rm min}^{-1}}{n} \Big( v_{\rm max}^{1/2} + 2 \alpha^2 (L+1) \Big) , \nonumber\\
			M_3 & = M_4 =  \frac{\beta_1^2 v_{\min}^{-1}}{n(1-\beta_1)^2} \Big( 16 v_{\rm max}^{1/2} + 32 \alpha^2 (L+1) + \alpha^2 L^2\Big), \nonumber \\
			M_5 & = \frac{G^2 v_{\rm min}^{-3} }{n}\Big( v_{\rm max}^{1/2} + 2 \alpha^2 (L+1) \Big) ,\nonumber\\
			M_6 & = \frac{4 v_{\rm min}^{-1}}{n}\Big( v_{\rm max}^{1/2} + 2 \alpha^2 (L+1) \Big) . \label{M_bound}
		\end{align}
	\end{Lemma}
	
	\begin{proof}
		The proof is given in Appendix B.
	\end{proof}
	
	Next, we focus on the following five error quantities on the right-hand side of (\ref{formula_lm_smooth}):
	\begin{itemize}
		\item[(i)] the consensus error $\|\bm{m}_{t} - \tilde{\bm{m}}_{t}\|^2$ of \emph{momentum gradient}; 
		\item[(ii)] the second-order moment $\|\tilde{\bm{m}}_{t}\|^2$ of average \emph{momentum gradient};
		\item[(iii)] the consensus error $\|\bm{x}_{t} - \tilde{\bm{x}}_{t}\|^2$ of \emph{local estimate};
		\item[(iv)] the consensus error $\|\bm{s}_{t} - \tilde{\bm{s}}_{t}\|^2$ of \emph{GT estimator}; 
		\item[(v)] the consensus error $\|\bm{v}_{t} - \tilde{\bm{v}}_{t}\|^2$ of \emph{adaptive vector}.
	\end{itemize}
	
	In the following Lemmas~\ref{lm_mcon}-\ref{lm_s_con}, we will establish the upper bounds on the expected summation over time horizon $T$ of the above five quantities respectively.

	\begin{Lemma}\label{lm_mcon} (Consensus error of momentum gradient) Consider the iterates generated by Algorithm~\ref{alg}. It holds for $T\geq 1$ that
		\begin{equation}\label{m_con2}
			\begin{split}
				& \sum_{t=1}^{T}\mathbb{E}[\|\bm{m}_{t} - \tilde{\bm{m}}_{t}\|^2] \leq 4 \sum_{t=1}^{T} \mathbb{E} [\| \bm{s}_{t} \!-\! \tilde{\bm{s}}_t\|^2].
			\end{split}
		\end{equation}
	\end{Lemma}
	\begin{proof}
		The proof is given in Appendix C-I.
	\end{proof}
	
	\begin{Lemma}\label{lm_x} (Consensus error of local estimate) Under Assumption~\ref{ass_graph}, the following inequality
		\begin{equation} \label{inequality_lm_x}
			\begin{split}
				&\sum_{t=1}^{T}\mathbb{E} [\|\bm{x}_{t} - \tilde{\bm{x}}_{t}\|^2] \leq \frac{ 40 \alpha^2 v_{\rm min}^{-1} }{(1 - \rho_A^2)^2}  \sum_{t=1}^{T} \mathbb{E} [\| \bm{s}_{t} - \tilde{\bm{s}}_t\|^2]\\
				& \qquad  + \frac{2}{1 - \rho_A^2} \Delta_1
			\end{split}
		\end{equation}
		holds for $T \geq 1$, where $\Delta_1 = \| \bm{x}_{1} - \tilde{\bm{x}}_{1} \|^2$ and
		\begin{equation}\label{rhoA}
			\rho_A = \rho \Big(A - \frac{\bm{1}_n \bm{1}_n'}{n}\Big).
		\end{equation}	
	\end{Lemma}
	\begin{proof}
		The proof is given in Appendix C-II.
	\end{proof}
	
	\begin{Lemma}\label{lm_mbar} (Norm of momentum gradient) Under Assumption~\ref{ass_smooth}, the following inequality holds for $T \geq 1$:
		\begin{equation} \label{ineq_lm4}
			\begin{split}
				& \sum_{t=1}^{T}\mathbb {E} \big[\|\tilde{\bm{m}}_{t} \|^2\big]  \leq \frac{ 240 L^2 \alpha^2 v_{\rm min}^{-1} }{(1 - \rho_A^2)^2}  \sum_{t=1}^{T} \mathbb{E} [\| \bm{s}_{t} \!-\! \tilde{\bm{s}}_t\|^2] \\
				& \quad +  6 {n} \sum_{t=1}^{T}  \mathbb {E}\big[\| \nabla f(\bar{x}_t) \|^2\big]  + 6 n \sigma^2 T  + \frac{12 L^2}{1 - \rho_A^2} \Delta_1. \\
			\end{split}
		\end{equation}		
	\end{Lemma}
	\begin{proof}
		The proof is given in Appendix C-III.
	\end{proof}
	
	\begin{Lemma}\label{lm_v_con} (Consensus error of adaptive vector) Under Assumption~\ref{assumption_sto_grad}(b), the following inequality
		\begin{equation} 
			\begin{split}
				\sum_{t=1}^{T} \mathbb{E} [\| \bm{v}_{t} - \tilde{\bm{v}}_{t}\|^2] \leq & \frac{36G^2 }{(1-\rho_A)^2} \sum_{t=1}^{T}  \mathbb{E} [\| \bm{s}_t - \tilde{\bm{s}}_t \|^2] \\
				&  + \frac{2}{1 - \beta_2^2} \Delta_2 
			\end{split}
		\end{equation}
	holds for $T \geq 1$, where $\Delta_2 = \| \bm{v}_{1} - \tilde{\bm{v}}_{1} \|^2$.
	\end{Lemma}
	\begin{proof}
		The proof is given in Appendix C-IV.
	\end{proof}

	\begin{Lemma}\label{lm_s_con} (Consensus error of GT estimator)	Under Assumptions~\ref{ass_smooth} and \ref{ass_graph}, for $\alpha^2 \leq \min \left\{\frac{1}{2 N_1} , \frac{v_{\min} (1-\rho_A^2)^2}{72} \right\}$ and $T \geq 1$ it holds that
		\begin{equation} \label{ineq_lm6}
			\begin{split}
				\sum_{t=1}^{T} \mathbb{E}\big[\|\bm{s}_{t} - \tilde{\bm{s}}_{t}\|^2\big] &\leq  2 n \alpha^2 N_2  \sum_{t=1}^{T}\mathbb{E} \big[\| \nabla f(\bar{x}_t)\|^2 \big]\\
				& \quad   + 2nT N_3  \sigma^2  + 2 N_4 \Delta,
			\end{split}
		\end{equation}
		where $\rho_A$ is given by~(\ref{rhoA}), $\Delta = \sum_{i=1}^{3} \Delta_i $ with $\Delta_1, \Delta_2$ given in Lemmas~\ref{lm_x} and \ref{lm_v_con}, respectively, while $\Delta_3 = \| \bm{s}_{1} - \tilde{\bm{s}}_{1}\| ^2$, and
		\begin{equation} \label{def_N1234}
			\begin{split}
				N_1 &= \frac{14400L^2 v_{\min}^{-1}}{(1-\rho_A^2)^4} + \frac{72 v_{\min}^{-1} }{(1 - \rho_A^2)^2} \left(4 + \frac{20 L^2}{3}\right) ,\\
				N_2 &= \frac{504  v_{\min}^{-1}}{(1 - \rho_A^2)^2} , \qquad N_3 = 12 + \frac{216}{(1 - \rho_A^2)^2} , \\
				N_4 &= \max \left\{ \frac{720 L^2}{(1 - \rho_A^2)^5} + \frac{24 L^2}{1 - \rho_A^2} , \frac{2 \beta_1^2}{1 - \beta_1^2}, \frac{1}{2} \right\}.
			\end{split}
		\end{equation}
	\end{Lemma}
	\begin{proof}
		The proof is given in Appendix C-V.
	\end{proof}
	
	Next, we present our main result, which provides an upper bound on the average gradient norm by Lemmas~\ref{lm_smooth}-\ref{lm_s_con}.
	\begin{Theorem} \label{The_1}
		Suppose Assumptions~\ref{ass_smooth}-\ref{ass_graph} hold, and there exists~$f^* \in \mathbb{R}$ such that $f(x) \geq f^*$ holds for all $x\in \mathbb{R}^d$. 
		For any positive scalar $\omega > 0$, by choosing parameter $\alpha$ such that $\alpha^2 <  \min \left\{\frac{1}{2 N_1} , \frac{v_{\min} (1-\rho_A^2)^2}{72}, \frac{v_{\max}^{-1}}{(N_2')^2} \right\}$ and $\beta_1$ such that $\frac{\beta_1^2}{\omega(1 - \beta_1)^2} \leq \alpha^2$, then we have
		\begin{equation}\label{gradient_bound2}
			\begin{split}
				& \frac{1}{T}\sum_{t=1}^{T}\mathbb{E} \big[\|\nabla f(\bar x_t)\|^2  \big] \leq O \bigg(\frac{f(\bar{x}_1) - f^* + \Delta}{T} + \sigma^2 \bigg),
			\end{split}
		\end{equation}
		where 
		\begin{equation} \label{def_N2_prime}
			\begin{split}
				N_2' &= N_2 \left[\frac{L}{9} \mu + \mu + \frac{72 G^3 v_{\min}^{-2}}{(1 - \rho_A)^2} \right]  + \frac{v_{\min}^{-1/2}}{2} + v_{\min}^{-1} (L+1) \\
				& \quad  + \omega \left( \mu + \frac{L^2}{36}\right) \left[3 + 4 N_2 (L^2 + 1)\right]
			\end{split}
		\end{equation} 
		with \begin{equation}
			\mu = 36 v_{\min}^{-1} v_{\max}^{1/2} + L + 1.
		\end{equation}
	\end{Theorem}
	
	\begin{proof}
		By summing~(\ref{formula_lm_smooth}) over $t=1,\dots, T$ and taking the expectation, one has
		\begin{equation} \label{the_inequality_1}
			\begin{split}
				& \mathbb{E}[f(\bar{z}_{T+1})] -f(\bar{z}_{1}) \\
				& \leq  - \alpha \Big( v_{\rm max} ^{-1/2} -  \frac{\alpha v_{\rm min}^{- 1/2}}{2} - \alpha v_{\rm min}^{-1} (L+ 1) \Big)  \sum_{t=1}^{T} \mathbb{E} [\|\nabla f(\bar{x}_t)\|^2] \\
				& \quad + M_1 \sum_{t=1}^{T} \mathbb{E} [\| {\bm{x}}_t - \tilde{\bm{x}}_t \|^2]  + M_2 \sum_{t=1}^{T} \mathbb{E} [\| \bm{s}_{t} - \tilde{\bm{s}}_t \|^2] \\
				& \quad + M_3 \sum_{t=1}^{T} \mathbb{E} [\|  \bm{m}_t - \tilde{\bm{m}}_t \|^2] + M_4 \sum_{t=1}^{T} \mathbb{E} [\| \tilde{\bm{m}}_t \|^2] \\
				& \quad+ M_5 \sum_{t=1}^{T} \mathbb{E} [\| \bm{v}_{t} - \tilde{\bm{v}}_t\|^2] + M_6 \sum_{t=1}^{T} \sum_{i=1}^{n} \| \eta_{t,i}\|^2.
			\end{split}
		\end{equation}
		
		Then, by Lemmas~\ref{lm_mcon}-\ref{lm_s_con} and the fact that $\alpha^2 \leq \frac{v_{\min} (1-\rho_A^2)^2}{72} $ and $\frac{\beta_1^2}{\omega (1 - \beta_1)^2} \leq \alpha^2$, we rearrange the terms and get
		\begin{equation} \label{The_inequatly_1}
			\begin{split}
				\mathbb{E} [f(\bar{z}_{T+1})] \leq& f(\bar{z}_{1}) - \alpha v_{\rm max} ^{-1/2} \sum_{t=1}^{T} \mathbb{E} [\|\nabla f(\bar{x}_t)\|^2] \\
				& + \alpha^2 N_2'  \sum_{t=1}^{T}\mathbb{E} \big[\| \nabla f(\bar{x}_t)\|^2 \big] + T N_3'  \sigma^2  + N_4' \Delta
			\end{split}
		\end{equation}
		with $N_2'$ given by~(\ref{def_N2_prime}), and
		\begin{equation} \label{def_N34_prime}
			\begin{split}
				N_3' &= \mu N_3\left[ \left(\frac{L}{9} + \omega (4 + 4 L^2) + 1 \right) \right] + N_3 \left(\frac{\omega L^2}{9}   + \frac{L^4}{10} \right) \\
				& \quad + \mu \left(\frac{1}{9} + 3\omega \right) + \frac{\omega L^2}{24},\\
				N_4' &= \frac{\mu N_4}{n} \left[\frac{10 L}{81 } + 1 + \omega \alpha^2 \left(8 + \frac{20 L^2}{3} \right) + \frac{72 G^3 v_{\min}^{-2}}{(1 - \rho_A)^2}\right] \\
				& \quad + \max \left\{ \frac{2 L \mu}{9 n (1 - \rho_A^2) } + \frac{12 \omega L^2\alpha^2}{n (1 - \rho_A^2)} ,    \frac{G^2 v_{\min}^{-2} \mu}{18 n (1 - \beta_2^2)}\right\}.
			\end{split}
		\end{equation}
		
		Note that the value of $N_2'$ is independent of $\alpha$ and $\beta_1$. 
		Then, by setting $0 < \alpha < \frac{v_{\max}^{-1/2}}{N_2'}$, which implies that $\alpha (v_{\max}^{-1/2} - \alpha N_2') > 0$, we can rearrange the terms in~(\ref{The_inequatly_1}) and obtain
		\begin{equation} \label{The_inequatly_2}
			\begin{split}
				& \frac{1}{T} \sum_{t=1}^{T} \mathbb{E} [\|\nabla f(\bar{x}_t)\|^2] \leq \frac{f(\bar{x}_{1}) - f^* + T N_3'  \sigma^2  + N_4' \Delta}{\alpha T (v_{\max}^{-1/2} - \alpha N_2')}
			\end{split}
		\end{equation}
		with $N_2'$ given in (\ref{def_N2_prime}) and $N_3', N_4'$ in~(\ref{def_N34_prime}). 
		The inequality~(\ref{The_inequatly_2}) implies that		
		\begin{equation}
			\begin{split}
				& \frac{1}{T}\sum_{t=1}^{T} \mathbb{E} [\|\nabla f(\bar{x}_t)\|^2] \leq {O} \left(   
				\frac{f(\bar{x}_{1}) - f^* + \Delta}{ T } + \sigma^2\right),
			\end{split}
		\end{equation}
		which completes the proof.
	\end{proof}
	
	Finally, by using Theorem~\ref{The_1}, Lemmas~\ref{lm_x} and \ref{lm_s_con}, we establish the upper bound on the optimality gap in the following corollary.
	\begin{Corollary} \label{Col_1}
		Under the conditions of Theorem~\ref{The_1}, it holds that
		\begin{equation}
			\begin{split}
				&\frac{1}{T}\sum_{t=1}^{T} \left[\mathbb{E} [\|\nabla f(\bar x_t)\|^2 + \frac{1}{n}\sum_{i=1}^{n}\|x_{t,i} - \bar{x}_t\|^2]\right]   \\
				& \leq \frac{ 80 N_2  \alpha^2 v_{\rm min}^{-1} }{\alpha (1 - \rho_A^2)^2 (v_{\max}^{-1/2} - \alpha N_2')} \cdot \frac{f(\bar{x}_{1}) - f^*}{T} \\
				& \quad +  \left[\frac{ 80 \alpha^2 v_{\rm min}^{-1} }{(1 - \rho_A^2)^2} \! \left(\frac{N_2 N_4'}{\alpha (v_{\max}^{-1/2} \!-\! \alpha N_2')} \!+\! \frac{N_4}{n} \right) \!+\!  \frac{2}{n (1 \!-\! \rho_A^2)}\right] \cdot \frac{\Delta}{T} \\
				& \quad + \frac{ 40 \alpha^2 v_{\rm min}^{-1} }{(1 - \rho_A^2)^2} \cdot  \left(\frac{ 2 N_2 N_3' }{\alpha (v_{\max}^{-1/2} - \alpha N_2')} + 2 N_3 \right) \sigma^2\\
				&    \leq O \bigg( \frac{f(\bar{x}_1) - f^* + \Delta}{T} + \sigma^2\bigg).
			\end{split}
		\end{equation}
	\end{Corollary}
\begin{proof}
			As we have already shown in Theorem~\ref{The_1} that $\frac{1}{T}\sum_{t=1}^{T} \mathbb{E} [\|\nabla f(\bar x_t)\|^2] \leq O \big( \frac{f(\bar{x}_1) - f^* + \Delta}{T} + \sigma^2\big)$, we then focus on the term $\frac{1}{nT}\sum_{t=1}^{T} \mathbb{E} [\sum_{i=1}^{n}\|x_{t,i} - \bar{x}_t\|^2]$. 
			
			By applying Lemmas~\ref{lm_x} and \ref{lm_s_con}, it is easy to prove that both $\sum_{t=1}^{T} \mathbb{E}\big[\|\bm{s}_{t+1} - \tilde{\bm{s}}_{t+1}\|^2\big]$ and $\sum_{t=1}^{T} \mathbb{E}\big[\|\bm{x}_{t+1} - \tilde{\bm{x}}_{t+1}\|^2\big]$ are of the order $O \big( \frac{f(\bar{x}_1) - f^* + \Delta}{T} + \sigma^2\big)$. 
			The detailed proof is provided in the Appendix~D.
\end{proof}
	
	\begin{Remark}
		As shown in Corollary~\ref{Col_1}, when the variance of the stochastic gradient is upper bounded by $\sigma^2$, the optimality gap of our proposed algorithm is in the order of $O(\frac{f(\bar{x}_1) - f^* + \Delta}{T} + \sigma^2)$, which is consistent with the one in centralized adaptive gradient algorithm~\cite{zaheer2018adaptive} for non-convex stochastic optimization with $L$-smooth cost functions. Note that the optimality gap will converge to some constant determined by the variance of the stochastic gradient. To mitigate the impact of the variance on the optimality gap, some variance reduction method~\cite{sun2020improving,jiang2021distributed,xin2020variance} can be further taken into consideration. 
	\end{Remark}
	
	\begin{Remark}
		Distributed optimization algorithms with adaptive stepsizes were developed in~\cite{zhang2021distributed} and~\cite{shen2020distributed} for convex optimization problems, in which the local gradients are used in the momentum-gradient descent. 
		Since the heterogeneity between different local cost functions sometimes leads to the disagreement on the momentum gradients as well as the adaptive stepsizes of different nodes, the algorithms in~\cite{zhang2021distributed} and~\cite{shen2020distributed} implement diminishing stepsizes to handle such disagreement. 
		However, the effect of gradient descent will become weaker as the stepsizes decay, which may affect the convergence performance.
		
		To mitigate the influence of heterogeneity, in our Algorithm~\ref{alg} we employ a GT estimator $\bm{s}_{t}$ to track the average stochastic gradient~$\tilde{\bm{s}}_{t}$. 
		It is proved that the convergence of the norm~$\|\bm{s}_{t} - \tilde{\bm{s}}_{t}\|^2$ will result in the convergence of the consensus errors of the local estimate~$\|\bm{x}_{t} - \tilde{\bm{x}}_{t}\|^2$ and the momentum vector~$\|\bm{m}_{t} - \tilde{\bm{m}}_{t}\|^2$, as shown in Lemmas~\ref{lm_mcon} and~\ref{lm_x}, respectively.
		Moreover, our Lemma~\ref{lm_v_con} further highlights that a smaller value of~$\|\bm{s}_{t} - \tilde{\bm{s}}_{t}\|^2$ also helps to reduce the disagreement on the adaptive vector~$\|\bm{v}_{t} - \tilde{\bm{v}}_{t}\|^2$, which plays a vital role in the stationarity analysis for our considered distributed stochastic non-convex optimization problem. 
		As a result, the utilization of the GT estimator enables our Algorithm~\ref{alg} to achieve a stationary performance even in the presence of heterogeneity, as illustrated in Corollary~\ref{Col_1}.
	\end{Remark}
		
	\section{Numerical examples}\label{sec_numerical}
	In this section, the performance of our proposed Algorithm~\ref{alg} is verified through numerical examples.
	
	\subsection{Distributed state estimation problem}
	
		We first consider a robust linear regression problem~\cite{yang2020graduated} based on the Huber loss function $H_{\varsigma}(\cdot) : \mathbb{R} \rightarrow \mathbb{R}^+$, defined by 
	\begin{equation}
		H_{\varsigma}(z)= 
		\begin{cases}\frac{1}{2} z^2, & \text { if }|z| \leq \varsigma, \\
			\varsigma\left(|z|-\frac{1}{2} \varsigma\right), & \text { otherwise }\end{cases}
	\end{equation}
	with scalar $\varsigma > 0$.
	
	In this example, the nodes in a network are to minimize the optimization problem given by~(\ref{problem}) with Huber-type local cost function $f_{i}:\mathbb{R}^m\rightarrow \mathbb{R}^{+}$ as follows~\cite{ghaderyan2023fast}:
	\begin{equation}
		f_{i}(x) = \mathbb{E} \left[ {\boldsymbol{H}}_{\varsigma} \left(\theta_i - \Phi_i x \right) \right],
	\end{equation}
	where the matrix $\Phi_i\in \mathbb{R}^{d \times d}$ and $\theta_i\in \mathbb{R}^{d}$ represents the estimate target, while the function $\boldsymbol{H}_{\varsigma}: \mathbb{R}^m \rightarrow \mathbb{R}^m$ takes the Huber loss on each entry of vector $z = [z_j]_{j=1}^m$, i.e.,  $\boldsymbol{H}_{\varsigma}(\boldsymbol{z})=\left[H_{\varsigma}\left(z_j\right)\right]_{j=1}^m$.
	
	We consider the problem setup with $n = 16$, $d = 10$, and use a randomly generated Erdős–Rényi (ER) graph $\mathcal{G}$ with probability $0.7$.
	We set the parameters of our proposed Algorithm~\ref{alg} as $\alpha = 0.01$, $\beta_1 = 0.9$, $\beta_2 = 0.999$, $v_{\rm max} = 100$ and $v_{\rm min} = 10^{-8}$. 
	For comparison, we consider the DSGD~\cite{assran2019stochastic}, distributed GT~\cite{zhang2019decentralized}, Momentum DSGD~\cite{yuan2021decentlam} and the distributed adaptive gradient algorithms proposed in~\cite{zhang2021distributed,shen2020distributed}.
	The constant stepsize in DSGD, distributed GT, Momentum DSGD (with $\beta_1 = 0.9$) is set as~$\alpha = 0.01$, while the diminishing stepsizes in adaptive gradient algorithms~\cite{zhang2021distributed,shen2020distributed} (with $\beta_1 = 0.9, \beta_2 = 0.999$) decay in the form of $\alpha_t = \frac{100 \alpha}{100 + \sqrt{t}}$.
	
	Fig.~\ref{fig_huber} illustrates the evolutions of the loss functions and the optimality gaps over $T = 2 \times 10^4$ iterations across $50$ random seeds. In the example, the optimal solution $x^*$ is randomly generated from $[-1, 1]^d$, and each matrix~$\Phi_i$ is generated with eigenvalues ranging in $[0.05, 1]$. Moreover, the noisy estimate target is generated following $\theta_i(t) = \Phi_i x^* + \eta_{t,i}$ with $\eta_{t,i}$ a random vector whose entries are sampled from truncated Gaussian distribution~\cite{chopin2011fast} with variance $0.04$ and truncated threshold~$0.1$. 
	The results demonstrate the superior performance of our proposed algorithm in terms of both the decay of the loss function and the optimality gap.
	We can observe that the DSGD and the momentum DSGD algorithms exhibit a fast rate at the beginning, however, these two algorithms converge to larger values of the loss function and the optimality gap as the iteration number increases. 
	On the other hand, the distributed adaptive gradient algorithms proposed in~\cite{zhang2021distributed} and~\cite{shen2020distributed} show a slower convergence due to their diminishing stepsizes.
	On the contrary, our Algorithm~1 exhibits a superior stationarity performance, displaying relatively smaller values of the loss function and the optimality gap as the iteration increases.

	\begin{figure}[!htb] 
		\centering
		\begin{subfigure}{0.24\textwidth}
			\centering   
			\includegraphics[width=1.0\linewidth,trim=100 250 120 280,clip]{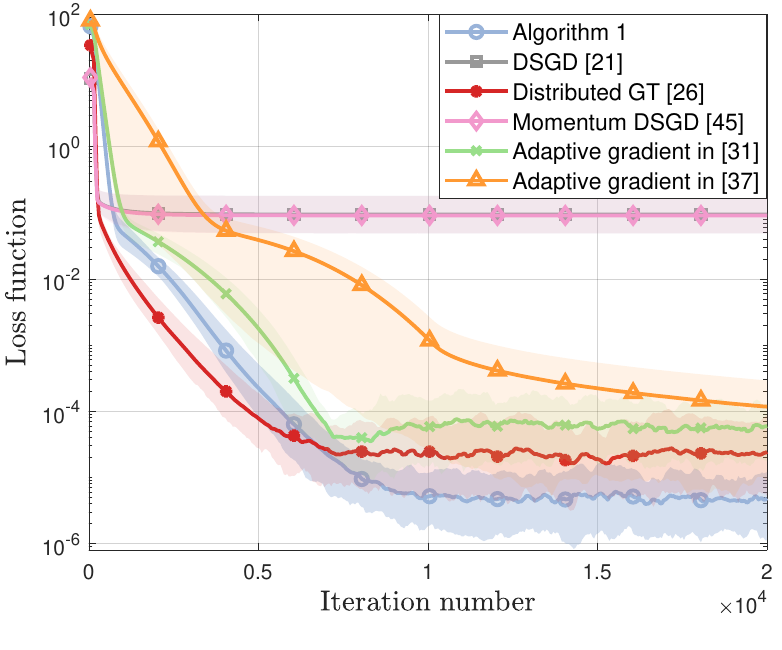}
			\caption{Loss functions}
			\label{fig_huber:loss}
		\end{subfigure} 
		\begin{subfigure}{0.24\textwidth}
			\centering   
			\includegraphics[width=1.0\linewidth,trim=100 250 120 280,clip]{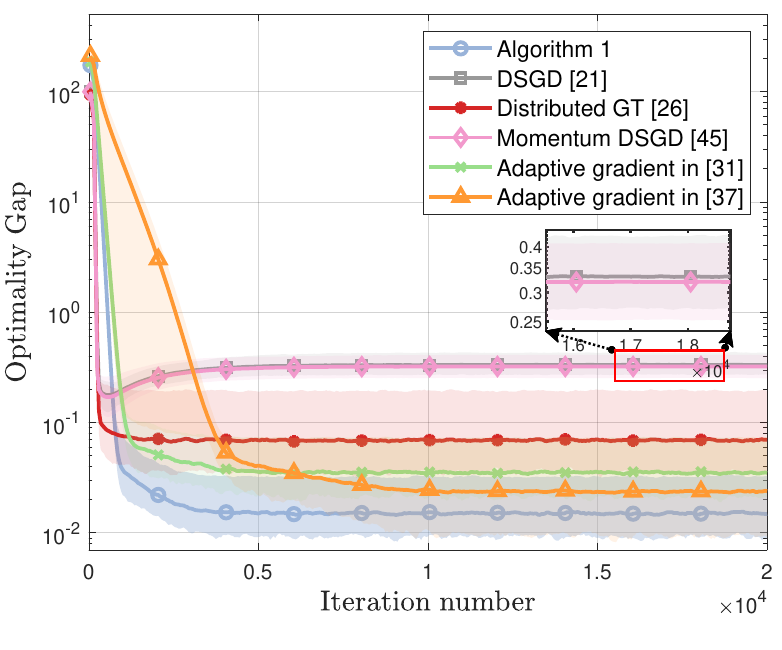}
			\caption{Optimality gaps}
			\label{fig_huber:gap}
		\end{subfigure}
		\caption{ Comparison of different algorithms on distributed robust linear regression problems. 
			Solid curves and shaded regions represent the average value and range statistics, respectively.}
		\label{fig_huber}
	\end{figure}	
	
	\subsection{Distributed logistic regression problem} \label{subsec_example1}
	Consider the logistic regression problem commonly arises in parameter estimation and machine learning. We present the examples based on three real-world datasets, i.e., the a9a, the Covertype, and the MNIST datasets\footnote[1]{The datasets are available at https://www.openml.org/}, where the data pieces belonging to~$c$ classes are distributed among~$n$ nodes.
	Denote the sampled data assigned to node $i$ as ${(y_{i,s}, l_{i,s}), s = 1,\dots, m}$, where $y_{i,s} \in \mathbb{R}^d$ represents the features and $l_{i,s} = [l_{i,s}^1 ,\dots , l_{i,s}^c] \in \{0,1\} \times \cdots \times \{0,1\}$ represents the label of the $s$-th sample at node $i$ using one-hot encoding. Consequently, the cost function of node $i$ can be expressed as follows: 
	\begin{equation} \label{simulation}
		\begin{split}
			f_{i}(W) \!=\! &- \! \frac{1}{m}\sum_{j=1}^{m}\sum_{k=1}^{c} l_{i,j}^{k} \log \! \Bigg( \frac{\exp(w_{
					k}' y_{i,j})}{\sum_{\hat{k}=1}^{c}\exp(w_{\hat{k}}' y_{i,j})} \Bigg) \!+\! h(W),
		\end{split}
	\end{equation}
	where $W = \big[w_1,\dots,w_{c}\big]' \in \mathbb{R}^{c\times d}$ is the weight of the logistic regression model to be optimized and $h(W) = \sum_{j_1 \leq c, j_2 \leq d} \frac{ 0.01 ([w_{j_1}]_{j_2})^2 }{1 + ([w_{j_1}]_{j_2})^2}$ is the non-convex Geman-McClure regularize function.
	
	\begin{figure}[!htb] 
		\centering
		\begin{subfigure}{0.24\textwidth}
			\centering   
			\includegraphics[width=1.0\linewidth,trim=100 250 110 280,clip]{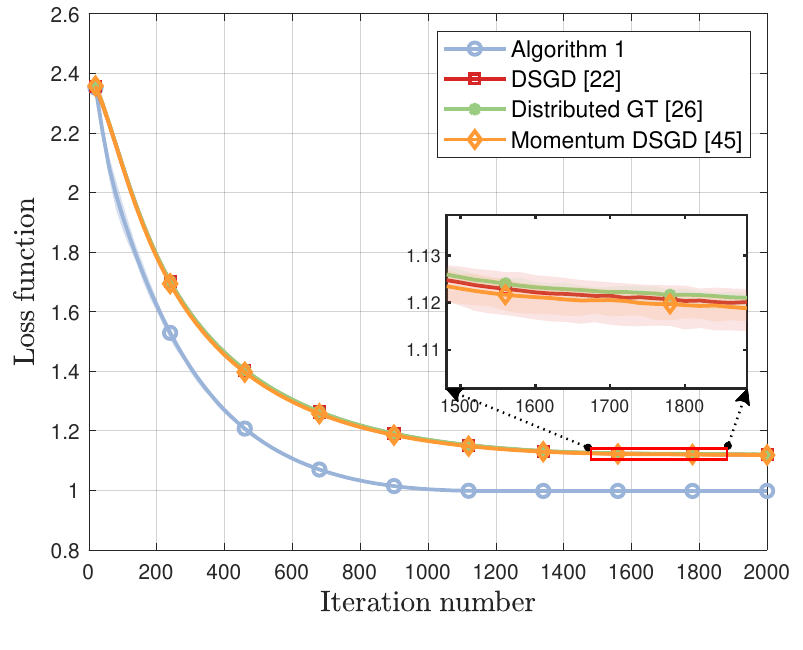}
			\caption{Loss functions}
			\label{fig_a9a:loss}
		\end{subfigure} 
		\begin{subfigure}{0.24\textwidth}
			\centering   
			\includegraphics[width=1.0\linewidth,trim=100 250 110 280,clip]{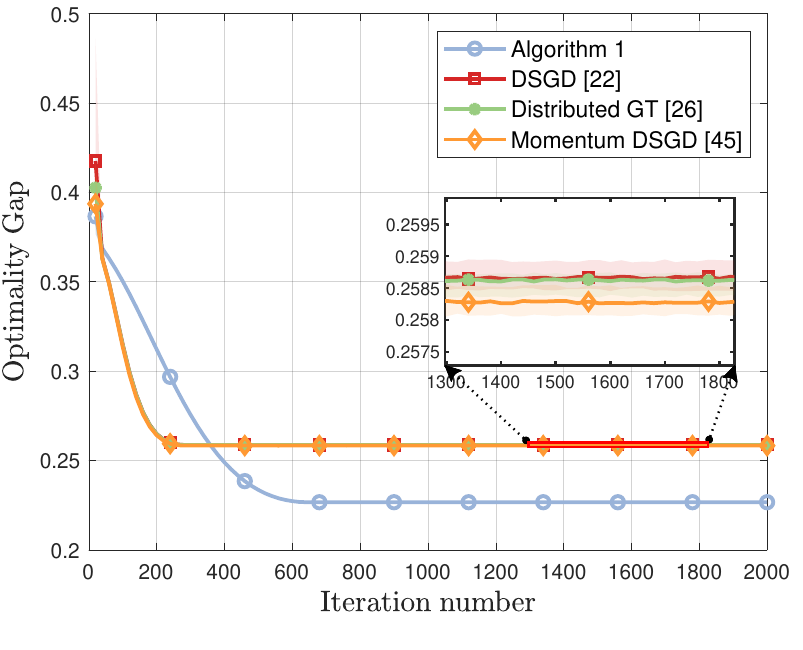}
			\caption{Optimality gaps}
			\label{fig_a9a:gap}
		\end{subfigure}
		\caption{ Comparison of different algorithms on distributed logistic regression on a9a dataset. }
		\label{fig_a9a}
	\end{figure}
	\begin{figure}[!htb] 
		\centering
		\begin{subfigure}{0.24\textwidth}
			\centering   
			\includegraphics[width=1.0\linewidth,trim=100 250 110 280,clip]{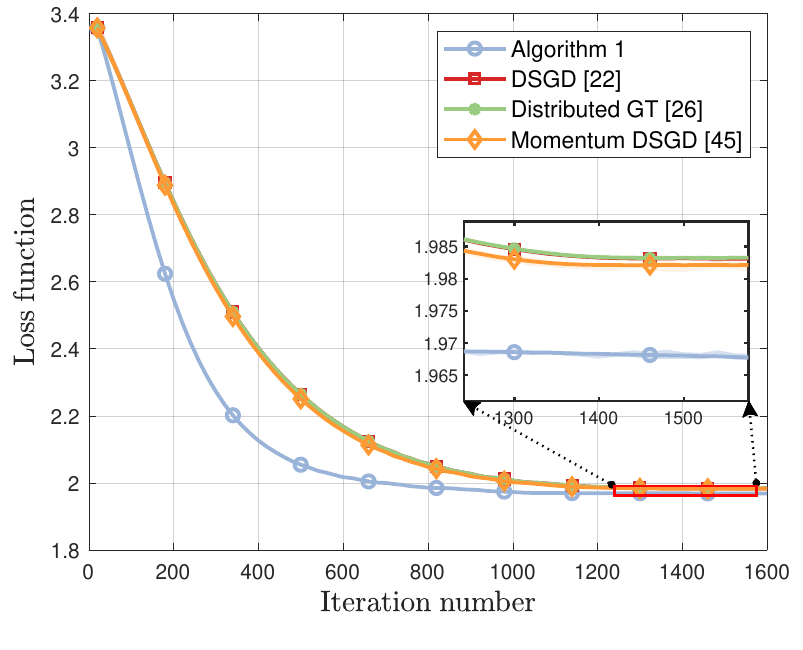}
			\caption{Loss functions}
			\label{fig_covertype:loss}
		\end{subfigure} 
		\begin{subfigure}{0.24\textwidth}
			\centering   
			\includegraphics[width=1.0\linewidth,trim=100 250 110 280,clip]{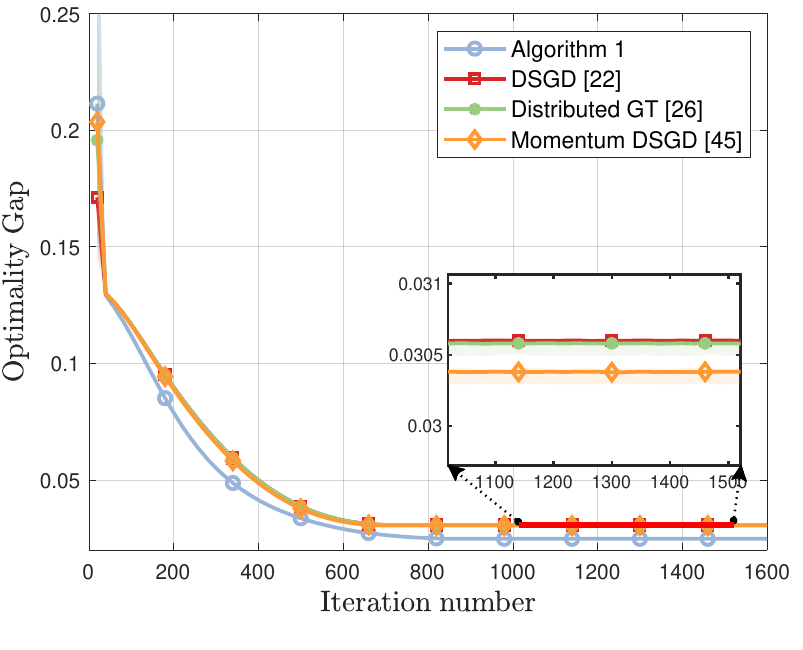}
			\caption{Optimality gaps}
			\label{fig_covertype:gap}
		\end{subfigure}
		\caption{ Comparison of different algorithms on distributed logistic regression on Covertype dataset. }
		\label{fig_covertype}
	\end{figure}	
	\begin{figure}[!htb] 
		\centering
		\begin{subfigure}{0.24\textwidth}
			\centering   
			\includegraphics[width=1.0\linewidth,trim=100 250 110 280,clip]{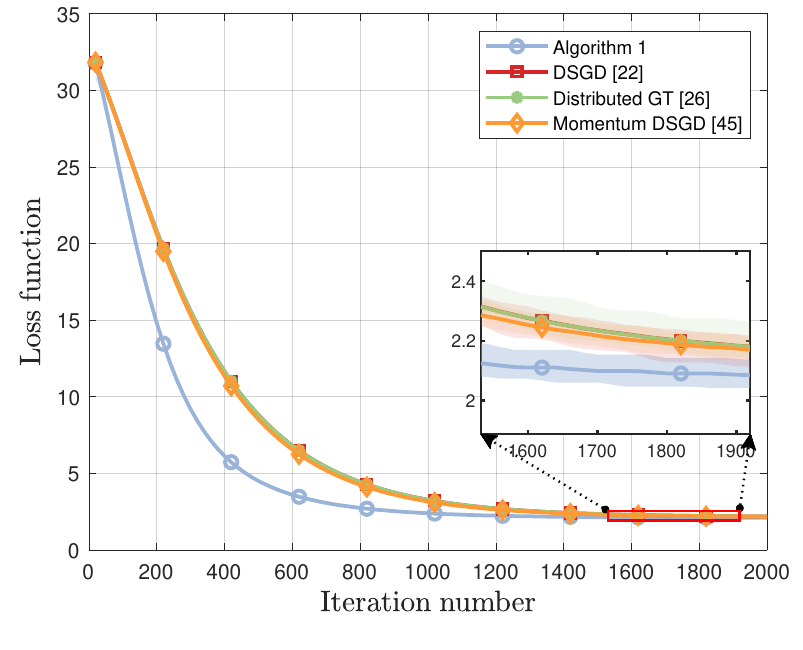}
			\caption{Loss functions}
			\label{fig_mnist:loss}
		\end{subfigure} 
		\begin{subfigure}{0.24\textwidth}
			\centering   
			\includegraphics[width=1.0\linewidth,trim=100 250 110 280,clip]{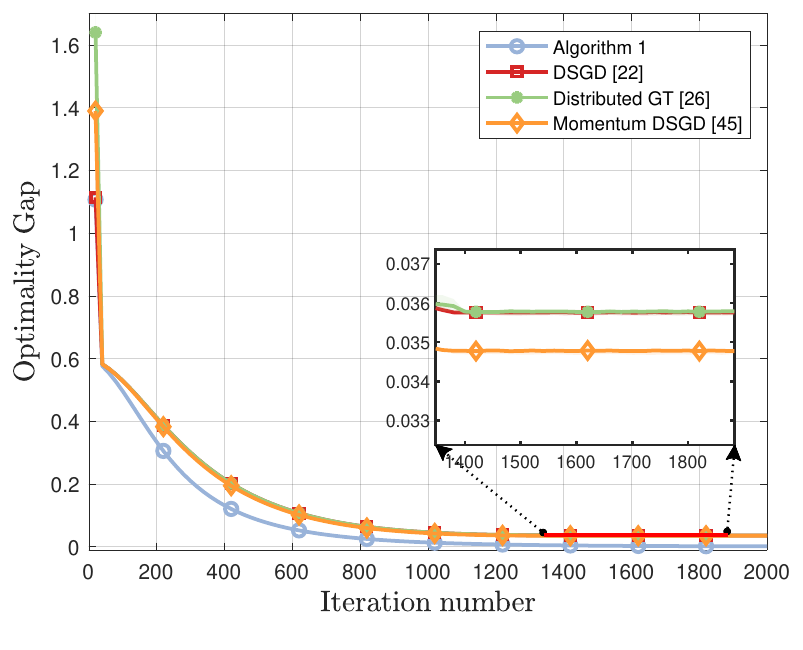}
			\caption{Optimality gaps}
			\label{fig_mnist:gap}
		\end{subfigure}
		\caption{ Comparison of different algorithms on distributed logistic regression on MNIST dataset. }
		\label{fig_mnist}
	\end{figure}	
	
	The datasets are evenly assigned to $n = 16$ nodes that can communicate through an ER graph. 
	The batch size and the iteration numbers are set to be $32$ and $2000$, respectively. 
	We set the parameters as $\alpha = 0.001$, $\beta_1 = 0.9$, $\beta_2 = 0.999$, $v_{\rm max} = 100$ and $v_{\rm min} = 10^{-8}$ for our Algorithm~\ref{alg}. 
	We also compare our algorithm with some state-of-the-art distributed non-convex optimization algorithms, including DSGD, distributed GT, and Momentum DSGD (with stepsizes $\alpha=0.1$).
	Numerical results conducted on the a9a, the Covtype and the MNIST datasets are presented in Figures \ref{fig_a9a}-\ref{fig_mnist}, respectively, with the number of random seeds set to $10$. 
	It is shown that our Algorithm~1 demonstrates the best performance in terms of the training loss and the optimality gap in these three distributed logistic regression tasks with non-convex regularizer.

	\section{Conclusion}
	In this paper, we investigated a distributed adaptive gradient algorithm for stochastic non-convex optimization problems with $L$-smooth cost functions, where a GT estimator is employed to handle the heterogeneity between different local cost functions, and the stepsizes are adjusted adaptively to enhance the performance of our algorithm when the gradients are sparse. 
	We analyzed the first-order stationarity performance of our proposed algorithm in terms of the optimality gap. 
	It is shown that an upper bound of the optimality gap is of the order $O(1/T + \sigma^2)$, which aligns with the one observed in centralized adaptive gradient algorithm~\cite{zaheer2018adaptive}. 
	The effectiveness of our proposed algorithm was validated through numerical examples. 
	Our future work includes research on distributed adaptive gradient algorithms with variance reduction and compressed communication for non-convex problems. 
	
	\bibliography{GTAdam}
	
	\section*{Appendix A: Auxiliary lemmas}
	This section presents some auxiliary lemmas that will be used in the proofs of Lemmas~1-6.
	\begin{Lemma}\label{lm_sqr}
		Consider two vectors $a,b\in\mathbb{R}^d$. The following inequality
		$$\|a+b\|^2 \leq (1+\theta)\|a\|^2 + (1+1/\theta)\|b\|^2$$
		holds for any $\theta>0$.
	\end{Lemma}
	\begin{proof}
		By Young's inequality, one has
		\begin{equation*}
			\|a+b\|^2 = \|a\|^2 + 2\langle a,b \rangle + \|b\|^2 \leq (1+\theta)\|a\|^2 + (1+1/\theta)\|b\|^2.
		\end{equation*}
		This completes the proof.
	\end{proof}	
	
	\begin{Lemma} \cite{xin2021stochastic} \label{ratio_sum} 
		Let $\{a_t\}$ and $\{b_t \}$ be scalar sequences such that
		$$a_{t+1} \leq q a_t + b_t$$
		with $ q \in (0,1)$. Then, for $T \geq 1$ we have
		$$		\sum_{t=1}^T a_t \leq \frac{1}{1-q} a_1+\frac{1}{1-q} \sum_{t=1}^{T} b_t	.	$$
	\end{Lemma}
	
	\begin{Lemma} [Gershgorin circle theorem] \cite{varga2010gervsgorin} \label{lemma_G_circle}  
		Let $A \in \mathbb{R}^{n \times n}$ be a complex matrix with entries $A_{ij}$. 
		Then, for any eigenvalue~$\lambda$ of matrix $A$, there exists a positive integer~$i \in \{1, \dots, n\}$ such that 
		$$ 	\left|\lambda-A_{i i}\right| \leq \sum_{j \neq i}\left|A_{i j}\right| .$$
	\end{Lemma}
	
	\begin{Lemma} \label{rho_bound}
		For the weighted adjacency matrix $A$ satisfying Assumption~\ref{ass_graph}, it holds that $\rho{(A - I_n)} \leq 2$.
	\end{Lemma}
	\begin{proof}
		From Lemma~\ref{lemma_G_circle}, the largest eigenvalue $\lambda_{\max} = \max_i |\lambda_i|$ of matrix $A - I_n$ satisfies
		$$- \sum_{j \neq i}\left|A_{i j}\right| \leq \lambda_{\max} - (A_{i i} - 1) \leq \sum_{j \neq i}\left|A_{i j}\right|,$$
		which implies that $\lambda_{\max} \leq 2$. This completes the proof.
	\end{proof}
	
	\begin{Lemma} \label{lm_basic_relation} (Basic relations of auxiliary vectors)
		Under Assumption~\ref{ass_graph}, the following equations
		\begin{align*}
			&(a)\;\bar{s}_t = \frac{1}{n} \sum\limits_{i=1}^{n} g_{t,i},\\
			&(b)\; \bar{x}_{t+1} - \bar{x}_t = - \frac{\alpha}{n} \sum\limits_{i=1}^{n} V_{t+1,i}^{-1/2} [\beta_1 {m}_{t,i} + (1-\beta_1) s_{t,i}],\\
			&(c)\; \bar z_{t+1} - \bar{z}_t = -\frac{\beta_1}{1-\beta_1} \frac{\alpha}{n} \sum\limits_{i=1}^{n} (V_{t+1,i}^{-1/2} - V_{t,i}^{-1/2}) m_{t,i} 
			\\
			&\qquad\qquad\qquad\;\; - \frac{\alpha}{n} \sum\limits_{i=1}^{n} V_{t+1,i}^{-1/2} s_{t,i}, \\
			&(d)\; \bar{z}_t - \bar{x}_t = -\frac{\beta_1}{1-\beta_1} \frac{\alpha}{n} \sum\limits_{i=1}^{n} V_{t,i}^{-1/2} m_{t,i}
		\end{align*}
		hold for $i\in \mathcal{V}$ and $t\geq 1$.
	\end{Lemma}
	
	\begin{proof}
		
		(a) When $t=1$, the equality holds by~(\ref{bar_def}) and the fact that $s_{1,i} = g_{1,i}$. 
		
		When $t\geq 2$, from~(\ref{s_update}), (\ref{bar_def}) and the fact that $A$ is doubly-stochastic, one has
		\begin{equation}
			\begin{split}
				\bar{s}_t &= \frac{1}{n} \sum_{i=1}^{n} s_{t,i} = \frac{1}{n} \sum_{i=1}^{n} \bigg[ \sum_{j=1}^{n} A_{ij} s_{t-1,j} + g_{t,i} - g_{t-1,i} \bigg] \\
				&= \bar{s}_{t-1} + \frac{1}{n} \sum_{i=1}^{n} (g_{t,i} - g_{t-1,i})\\
				&= \dots \\
				&= \bar{s}_{0} + \frac{1}{n} \sum_{i=1}^{n} (g_{t,i} - g_{0,i})\\
				&= \frac{1}{n} \sum_{i=1}^{n} g_{t,i},
			\end{split}
		\end{equation}
		where the last equation follows from~(\ref{bar_def}) and the fact that $s_{1,i} = g_{1,i}$.
		
		(b) By (\ref{m_update}), (\ref{x_update}) and (\ref{bar_def}), one has
		\begin{equation}\label{lm1b}
			\begin{split}
				\bar{x}_{t+1} &= \bar{x}_t - \frac{\alpha}{n} \sum_{i=1}^{n} V_{t+1,i}^{-1/2} {m}_{t+1,i} \\
				& = \bar{x}_t - \frac{\alpha}{n} \sum_{i=1}^{n} V_{t+1,i}^{-1/2} [\beta_1 {m}_{t,i} + (1-\beta_1) s_{t,i}],
			\end{split}
		\end{equation}
		which completes the proof.	
		
		(c) When $t=1$, the equality holds by (\ref{m_update}), (\ref{x_update}), (\ref{z_def}) and the fact $m_{1,i} = \bm{0}_d$, $i\in \mathcal{V}$.
		
		When $t\geq 2$, from~(\ref{z_def}) and (\ref{bar_def}), one has 
		\begin{equation}\label{lm1c}
			\bar{z}_t = \frac{1}{1-\beta_1} \bar{x}_t - \frac{\beta_1}{1-\beta_1} \bar x_{t-1}.
		\end{equation} 
		By~(\ref{lm1b}) and (\ref{lm1c}), we can obtain
		\begin{equation}\label{lm1c2}
			\begin{split}
				\bar{z}_{t+1} - \bar{z}_t &= \frac{1}{1-\beta_1} (\bar{x}_{t+1} - \bar{x}_t) - \frac{\beta_1}{1-\beta_1} (\bar x_{t} - \bar x_{t-1})\\
				& = - \frac{1}{1-\beta_1} \frac{\alpha}{n} \sum_{i=1}^{n} V_{t+1,i}^{-1/2} [\beta_1 {m}_{t,i} + (1-\beta_1) s_{t,i} ]\\
				& \quad + \frac{\beta_1}{1-\beta_1} \frac{\alpha}{n} \sum_{i=1}^{n} V_{t,i}^{-1/2} {m}_{t,i}.
			\end{split}
		\end{equation}
		The proof is completed by rearranging the terms in~(\ref{lm1c2}).
		
		(d) When $t=1$, the equality holds since $z_{1,i} = x_{1,i}$ and $m_{1,i} = \bm{0}_d$, $i\in \mathcal{V}$. 
		
		When $t\geq 2$, it follows from~(\ref{lm1b}) and (\ref{lm1c}) that
		\begin{equation}\label{zx}
			\begin{split}
				\bar{z}_t - \bar{x}_t &= \frac{\beta_1}{1-\beta_1} (\bar{x}_t - \bar{x}_{t-1}) \\
				&= -\frac{\beta_1}{1-\beta_1} \frac{\alpha}{n} \sum_{i=1}^{n} V_{t,i}^{-1/2} m_{t,i} .
			\end{split}
		\end{equation}
		This completes the proof.
	\end{proof}
	
	\section*{Appendix B: Proof of Lemma~\ref{lm_smooth}}
	\begin{proof}
		It follows from Assumption~\ref{ass_smooth} that
		\begin{equation}\label{smooth_func}
			\begin{split}
				f(\bar{z}_{t+1}) & \leq f(\bar{z}_{t}) - \langle \nabla f(\bar{z}_{t}) , \bar{z}_{t} - \bar{z}_{t+1}\rangle + \frac{L}{2} \|\bar{z}_{t} - \bar{z}_{t+1}\|^2\\
				& \leq f(\bar{z}_{t}) - \langle \nabla f(\bar{z}_{t}) - \nabla f(\bar{x}_{t}) , \bar{z}_{t}  - \bar{z}_{t+1}\rangle \\
				& \quad- \langle \nabla f(\bar{x}_{t}) , \bar{z}_{t} - \bar{z}_{t+1}\rangle + \frac{L}{2} \|\bar{z}_{t} - \bar{z}_{t+1}\|^2\\
				& \leq f(\bar{z}_{t}) + \frac{1}{2} \|\nabla f(\bar{z}_{t}) - \nabla f(\bar{x}_{t})\|^2  + \frac{1}{2}\|\bar{z}_{t} - \bar{z}_{t+1}\|^2 \\
				& \quad- \langle \nabla f(\bar{x}_{t}) , \bar{z}_{t} - \bar{z}_{t+1}\rangle + \frac{L}{2} \|\bar{z}_{t} - \bar{z}_{t+1}\|^2\\
				&\leq f(\bar{z}_{t}) + \phi_1(t) + \phi_2(t) + \phi_3(t),
			\end{split}
		\end{equation}
		where \begin{equation}
			\begin{split}
				\phi_1(t) &= \frac{L^2}{2} \| \bar{z}_{t} - \bar{x}_{t} \|^2, \\
				\phi_2(t) &= - \langle \nabla f(\bar{x}_{t}) , \bar{z}_{t} - \bar{z}_{t+1}\rangle, \\
				\phi_3(t) &= \frac{L+1}{2} \|\bar{z}_{t} - \bar{z}_{t+1}\|^2.
			\end{split}
		\end{equation}
		
		From (\ref{v3_update}) and (\ref{v2_update}), it is easy to show that $v_{\rm min}\leq [{V}_{t,i}]_{jj} \leq v_{\rm max}$ and $v_{\rm min} \leq [\bar{V}_t]_{jj} \leq v_{\rm max}$ hold for $j = 1,\dots,d$. 
		Then, for the term $\phi_1(t)$, it follows from Lemma~\ref{lm_basic_relation}(d) that 
		\begin{equation}\label{phi1}
			\begin{split}
				\phi_1(t)  &\leq  \Big(\frac{\alpha \beta_1 L}{1-\beta_1}\Big)^2 \frac{v_{\rm min}^{-1}}{2n}\sum_{i=1}^{n} \|m_{t,i}\|^2\\
				& = \Big(\frac{\alpha \beta_1 L}{1-\beta_1}\Big)^2 \frac{v_{\rm min}^{-1}}{2n} \|\bm{m}_{t}\|^2 \\
				&\leq \Big(\frac{\alpha \beta_1 L}{1-\beta_1}\Big)^2 \frac{v_{\rm min}^{-1}}{n} (\|\bm{m}_{t} - \tilde{\bm{m}}_{t}\|^2 + \|\tilde{\bm{m}}_{t}\|^2).
			\end{split}
		\end{equation}
		
		From Lemma~\ref{lm_basic_relation}(c) and the fact that $\bar{V}_{t+1}$ is a diagonal matrix, one has
		\begin{equation}\label{phi2}
			\begin{split}
				\phi_2(t) &= - \Big\langle \nabla f(\bar{x}_t) , \alpha \bar{V}_{t+1}^{-1/2} \nabla f(\bar{x}_t) \Big\rangle  + \Big\langle \alpha \bar{V}_{t+1}^{-1/4} \nabla f(\bar{x}_t) ,\\
				&\qquad \; \frac{1}{\alpha}\bar{V}_{t+1}^{1/4}(\bar{z}_{t+1} - \bar{z}_t) + \bar{V}_{t+1}^{-1/4} \nabla f(\bar{x}_t) \Big\rangle\\
				& \leq - \alpha v_{\rm max} ^{-1/2} \|\nabla f(\bar{x}_t)\|^2 + \frac{\alpha^2 v_{\rm min}^{-1/2}}{2} \|\nabla f(\bar{x}_t)\|^2 \\
				& \quad + \frac{ v_{\rm max}^{1/2}}{2} \Big\|\frac{\bar{z}_{t} - \bar{z}_{t+1}}{\alpha} - \bar{V}_{t+1} ^{-1/2} \nabla f(\bar{x}_t)\Big\|^2 \\
				& = - \alpha\Big( v_{\rm max} ^{-1/2} - \frac{\alpha v_{\rm min}^{-1/2}}{2}\Big)\|\nabla f(\bar{x}_t)\|^2 \\
				& \quad+ \frac{ v_{\rm max}^{1/2}}{2} \Big\|\frac{\beta_1}{1-\beta_1} \frac{1}{n} \sum_{i=1}^{n} ( V_{t+1,i}^{-1/2} -  V_{t,i}^{-1/2}) m_{t,i} \\
				& \qquad \qquad + \frac{1}{n} \sum_{i=1}^{n}  V_{t+1,i}^{-1/2} s_{t,i}  - \bar{V}_{t+1} ^{-1/2} \nabla f(\bar{x}_t)\Big\|^2\\
				& \leq - \alpha\Big( v_{\rm max} ^{-1/2} - \frac{\alpha v_{\rm min}^{-1/2}}{2}\Big) \|\nabla f(\bar{x}_t)\|^2 \\
				& \quad + 2 v_{\rm max}^{1/2}\big( \phi_{21}(t) + \phi_{22}(t) + \phi_{23}(t) + \phi_{24}(t)\big),
			\end{split}
		\end{equation}
		where \begin{equation} \label{phi1_4}
			\begin{split}
				\phi_{21}(t) &= \Big\|\frac{\beta_1}{1-\beta_1} \frac{1}{n} \sum_{i=1}^{n} ( V_{t+1,i}^{-1/2} -  V_{t,i}^{-1/2}) m_{t,i}\Big\|^2, \\
				\phi_{22}(t) &= \Big\|\frac{1}{n} \sum_{i=1}^{n}  (V_{t+1,i}^{-1/2} - \bar V_{t+1}^{-1/2}) s_{t,i} \Big\|^2, \\
				\phi_{23}(t) &= \Big\|\frac{1}{n} \sum_{i=1}^{n} \bar V_{t+1}^{-1/2} (s_{t,i} - \bar{s}_t) \Big\|^2, \\
				\phi_{24}(t) &= \Big\| \bar V_{t+1}^{-1/2} (\bar{s}_t - \nabla f(\bar{x}_t)) \Big\|^2.
			\end{split}
		\end{equation}
		
		Next, we focus on the upper bounds of the terms in~(\ref{phi1_4}). 
		For the term $\phi_{21}(t)$, since $V_{t+1,i}$ is a diagonal matrix with $v_{\rm min}\leq [{V}_{t,i}]_{jj}$, one has
		\begin{equation} \label{phi21}
			\begin{split}
				&\phi_{21}(t)\leq \Big(\frac{\beta_1}{1-\beta_1}\Big)^2 \frac{1}{n}\sum_{i=1}^{n} \Big\| V_{t+1,i}^{-1/2} - V_{t,i}^{-1/2} \Big\|^2 \|m_{t,i}\|^2 \\
				&\leq \Big(\frac{\beta_1}{1-\beta_1}\Big)^2 \frac{2}{n}\sum_{i=1}^{n} \Big(\big\| V_{t+1,i}^{-1/2} \big\|^2 + \big\|V_{t,i}^{-1/2} \big\|^2\Big) \|m_{t,i}\|^2 \\			
				& \leq \frac{4\beta_1^2 v_{\rm min}^{-1}}{n(1-\beta_1)^2} \|\bm{m}_t\|^2 \\
				&\leq \frac{8\beta_1^2 v_{\rm min}^{-1}}{n(1-\beta_1)^2} (\|{\bm{m}}_t - {\bm{\tilde m}}_t\|^2 + \|{\bm{\tilde m}}_t\|^2).
			\end{split}
		\end{equation}
		
		For the term $\phi_{22}(t)$, we have
		\begin{equation}\label{phi22_temp}
			\begin{split}
				\phi_{22}(t) & \leq \frac{1}{n}\sum_{i=1}^{n} \Big\| V_{t+1,i}^{-1/2} - \bar V_{t+1}^{-1/2} \Big\|^2 \|s_{t,i}\|^2 \\
				& \leq \frac{2}{n}\sum_{i=1}^{n} \Big\| V_{t+1,i}^{-1/2} - \bar V_{t+1}^{-1/2} \Big\|^2 (\|s_{t,i} - \bar{s}_t\|^2 + \|\bar{s}_t\|^2) \\
				&\leq \frac{8 v_{\rm min}^{-1}}{n} \|{\bm{s}}_t - {\bm{\bar s}}_t\|^2 + \frac{2}{n}\sum_{i=1}^{n} \Big\| V_{t+1,i}^{-1/2} - \bar V_{t+1}^{-1/2} \Big\|^2  \|\bar{s}_t\|^2.
			\end{split}
		\end{equation}
		Since $V_{t,i}$ and ${\bar V}_{t}$ are diagonal matrices, the following inequality
		\begin{equation}\label{sqr_V}
			\begin{split}
				&\big| [V_{t,i}]_{jj}^{-1/2} - [{\bar V}_{t}]_{jj}^{-1/2} \big| = \Bigg| \frac{[{\bar V}_{t}]_{jj}^{1/2} - {[V_{t,i}]_{jj}^{1/2}}}{{[{\bar V}_{t}]_{jj}^{1/2} [V_{t,i}]_{jj}^{1/2}}} \Bigg| \\
				& = \Bigg| \frac{([{\bar V}_{t}]_{jj}^{1/2} - {[V_{t,i}]_{jj}^{1/2}})({[{\bar V}_{t}]_{jj}^{1/2}} + {[V_{t,i}]_{jj}}^{1/2})}  {{[{\bar V}_{t}]_{jj}^{1/2} [V_{t,i}]_{jj}^{1/2}} ({[{\bar V}_{t}]_{jj}^{1/2}} + {[V_{t,i}]_{jj}}^{1/2})} \Bigg| \\
				& \leq \frac{v_{\rm min}^{-3/2}}{2} | [{\bar V}_{t}]_{jj} - {[V_{t,i}]_{jj}}|
			\end{split}
		\end{equation}
		holds for $j = 1,\dots,d$. 
		Then, from Lemma~\ref{lm_basic_relation}(a) and Assumption~\ref{assumption_sto_grad}(b), the second term on the right-hand side of (\ref{phi22_temp}) is bounded as follows:
		\begin{equation}\label{phi22_2}
			\begin{split}
				&\frac{2}{n}\sum_{i=1}^{n} \Big\| V_{t+1,i}^{-1/2} - \bar{V}_{t+1}^{-1/2} \Big\|^2  \|\bar{s}_t\|^2\\
				&\leq \frac{2G^2}{n}\sum_{i=1}^{n} \big\| V_{t+1,i}^{-1/2} - \bar{V}_{t+1}^{-1/2} \big\|_F^2 \\
				& \leq \frac{G^2 v_{\rm min}^{-3}}{2n} \sum_{i=1}^{n} \big\| V_{t+1,i} - \bar{V}_{t+1} \big\|_F^2  \\
				& = \frac{G^2 v_{\rm min}^{-3}}{2n}\| \bm{v}_{t+1} - \tilde{\bm{v}}_{t+1}\|^2 .
			\end{split}
		\end{equation}
		According to (\ref{phi22_temp}) and (\ref{phi22_2}), we have
		\begin{equation}\label{phi22}
			\begin{split}
				\phi_{22}(t) \leq \frac{8 v_{\rm min}^{-1}}{n} \|{\bm{s}}_t - {\bm{\bar s}}_t\|^2 + \frac{G^2 v_{\rm min}^{-3}}{2n}\| \bm{v}_{t+1} - \tilde{\bm{v}}_{t+1}\|^2 .
			\end{split}
		\end{equation}
		
		For the term $\phi_{23}(t)$, it holds that
		\begin{equation}\label{phi23}
			\begin{split}
				\phi_{23}(t) &=	\Big\|\frac{1}{n} \sum_{i=1}^{n} \bar V_{t+1}^{-1/2} (s_{t,i} - \bar{s}_t) \Big\|^2 \\
				&\leq \frac{v_{\rm min}^{-1}}{n} \sum_{i=1}^{n}\| s_{t,i} - \bar{s}_t \|^2  = \frac{v_{\rm min}^{-1}}{n} \| \bm{s}_{t} - \tilde{\bm{s}}_{t} \|^2.
			\end{split}
		\end{equation}
		
		Moreover, it follows from~(\ref{bar_def}) and Assumption~\ref{ass_smooth} that
		\begin{equation}\label{phi24}
			\begin{split}
				\phi_{24}(t) &=  \Big\| \bar V_{t+1}^{-1/2} (\bar{s}_t - \nabla f(\bar{x}_t)) \Big\|^2 \\
				&\leq v_{\rm min}^{-1} \bigg\| \frac{1}{n}\sum_{i=1}^{n} g_{t,i} - \nabla f(\bar{x}_t) \bigg\|^2 \\
				&= v_{\rm min}^{-1} \bigg\| \frac{1}{n}\sum_{i=1}^{n} (\nabla f_i({x}_{t,i}) + \eta_{t,i} - \nabla f_i(\bar{x}_t)) \bigg\|^2 \\
				& \leq \frac{2 v_{\rm min}^{-1} L^2}{n}\sum_{i=1}^{n} \|  {x}_{t,i} - \bar{x}_t \|^2 + \frac{2 v_{\rm min}^{-1} }{n}\sum_{i=1}^{n} \| \eta_{t,i}\|^2\\
				&= \frac{2v_{\rm min}^{-1} L^2}{n} \| \bm{x}_t - \tilde{\bm{x}}_{t} \|^2 + \frac{2 v_{\rm min}^{-1} }{n}\sum_{i=1}^{n} \| \eta_{t,i}\|^2.
			\end{split}
		\end{equation}
		
		Considering the term $\phi_3(t)$, one has
		\begin{equation}\label{phi3}
			\begin{split}
				\phi_3(t) & \leq \alpha^2 (L+1)  \big\| \bar{V}_{t+1} ^{-1/2} \nabla f(\bar{x}_t) \big\|^2\\
				&\quad + \alpha^2 (L+1)  \Big\|\frac{\bar{z}_{t} - \bar{z}_{t+1}}{\alpha} - \bar{V}_{t+1} ^{-1/2} \nabla f(\bar{x}_t) \Big\|^2 \\
				& \leq \alpha^2 (L+1)  v_{\rm min}^{-1} \|\nabla f(\bar{x}_t)\|^2 \\
				& \quad + 4 \alpha^2 (L \!+\! 1)  \big( \phi_{21}(t) \!+\! \phi_{22}(t) \!+\! \phi_{23}(t) \!+\! \phi_{24}(t)\big)
			\end{split}
		\end{equation}
		with $\phi_{21}(t),\phi_{22}(t),\phi_{23}(t),\phi_{24}(t)$ given by~(\ref{phi1_4}). 
		
		Finally, we complete the proof by (\ref{smooth_func})-(\ref{phi3}).

	\end{proof}
	
	\section*{Appendix C: Proofs of Lemmas~\ref{lm_mcon}-\ref{lm_s_con}}
	
	\begin{proof}[I: Proof of Lemma~\ref{lm_mcon}] 
%
		From the update rule~(\ref{m_update}), for any $t \geq 1$ one has 
		\begin{equation}\label{m_con}
			\begin{split}
				&\|\bm{m}_{t+1} - \tilde{\bm{m}}_{t+1}\| \\
				&= \big\| \beta_1 \bm{m}_{t} + (1-\beta_1) \bm{s}_{t} - (\beta_1 \tilde{\bm{m}}_{t} + (1-\beta_1) \tilde{\bm{s}}_t) \big\| \\
				&\leq \beta_1 \| \bm{m}_{t} - \tilde{\bm{m}}_{t} \| + (1-\beta_1) \| \bm{s}_{t} - \tilde{\bm{s}}_t\|.
			\end{split}
		\end{equation}
		By~(\ref{m_con}) and Lemma~\ref{lm_sqr}, it holds for $\theta = \frac{1 - \beta_1^2}{2 \beta_1^2} > 0$ that 
		\begin{equation} \label{ineq_lm2}
			\begin{split}
				& \mathbb{E}[\|\bm{m}_{t+1} - \tilde{\bm{m}}_{t+1}\|^2] \leq \beta_1^2 (1+\theta) \mathbb{E} [\| \bm{m}_{t} - \tilde{\bm{m}}_{t} \|^2]  \\
				& \quad + (1-\beta_1)^2 (1+1/\theta) \mathbb{E} [\| \bm{s}_{t} - \tilde{\bm{s}}_t\|^2]\\
				&= \frac{1 + \beta_1^2}{2}  \mathbb{E} [\| \bm{m}_{t} - \tilde{\bm{m}}_{t} \|^2] \\
				& \quad + \frac{1+ \beta_1^2}{1 - \beta_1^2} (1 - \beta_1)^2  \mathbb{E} [\| \bm{s}_{t} - \tilde{\bm{s}}_t\|^2].
			\end{split}
		\end{equation}
			
		For $T \geq 1$, it follows from~(\ref{ineq_lm2}) and Lemma~\ref{ratio_sum} that
		\begin{equation}
			\begin{split}
				& \sum_{t=1}^{T}\mathbb{E}[\|\bm{m}_{t} - \tilde{\bm{m}}_{t}\|^2] \leq \frac{2}{1 - \beta_1^2}  \|\bm{m}_{1} - \tilde{\bm{m}}_{1}\|^2\\
				& \qquad + \frac{2 (1+ \beta_1^2)}{(1 - \beta_1^2)^2} (1 - \beta_1)^2  \sum_{t=1}^{T} \mathbb{E} [\| \bm{s}_{t} - \tilde{\bm{s}}_t\|^2] \\
				& \leq 4 \sum_{t=1}^{T} \mathbb{E} [\| \bm{s}_{t} - \tilde{\bm{s}}_t\|^2],
			\end{split}
		\end{equation}
		where the second inequality holds since $m_{1,i} = \bm{0}_d$ and $\frac{(1 - \beta_1)^2}{(1 - \beta_1^2)^2} \leq 1$ as $\beta_1 \in (0, 1)$. This completes the proof.
	\end{proof}
	
	\begin{proof}[II: Proof of Lemma~\ref{lm_x}]
		From (\ref{x_update}), (\ref{bar_def}) and the fact that $[V_{t+1}]_{jj} \geq v_{\rm min}$, we have
		\begin{equation}\label{x}
			\begin{split}
				&\|\bm{x}_{t+1} - \tilde{\bm{x}}_{t+1}\| = \bigg\| (A \otimes I_d)\bm{x}_{t} - \alpha V_{t+1}^{-1/2} \bm{m}_{t+1} \\
				& \qquad \;  - \bigg[\Big( \frac{\bm{1}_n\bm{1}_n'}{n}  \otimes I_d \Big) {\bm{x}}_t - \Big(\frac{\bm{1}_n\bm{1}_n'}{n}\otimes I_d \Big) \alpha V_{t+1}^{-1/2} \bm{m}_{t+1} \bigg] \bigg\| \\
				& \leq \bigg\| \Big[\Big(A \!-\! \frac{\bm{1}_n\bm{1}_n'}{n}  \Big) \!\otimes\! I_d \Big](\bm{x}_{t} \!-\! \tilde{\bm{x}}_t) \bigg\| \!+\! \alpha v_{\rm min}^{-1/2} \| \bm{m}_{t+1} \!-\! \tilde{\bm{m}}_{t+1} \|\\
				& \leq \rho_A \| \bm{x}_{t} - \tilde{\bm{x}}_t\| + \alpha v_{\rm min}^{-1/2}  (1-\beta_1) \| \bm{s}_{t} - \tilde{\bm{s}}_{t} \|  \\
				& \quad + \alpha v_{\rm min}^{-1/2} \beta_1 \| \bm{m}_{t} - \tilde{\bm{m}}_{t}\|,
			\end{split}
		\end{equation}
		where the first inequality holds since $\big[\big(A - \frac{\bm{1}_n\bm{1}_n'}{n}  \big)\otimes I_d \big] \tilde{\bm{x}}_t = \bm{0}_{nd}$, and the last one follows from (\ref{m_con}).
		
		Then, from (\ref{x}), we can use Lemma~\ref{lm_sqr} with $\theta = \frac{1 - \rho_A^2}{2 \rho_A^2}$ and take the total expectation to obtain
		\begin{equation}
			\begin{split}
				&\mathbb{E} [\|\bm{x}_{t+1} - \tilde{\bm{x}}_{t+1}\|^2] \leq \frac{1 + \rho_A^2}{2} \mathbb{E} [\| \bm{x}_{t} - \tilde{\bm{x}}_t\|^2] \\
				& \qquad \quad + 2 \alpha^2 v_{\rm min}^{-1} (1-\beta_1)^2 \frac{1 + \rho_A^2}{1 - \rho_A^2} \mathbb{E} [\| \bm{s}_{t} - \tilde{\bm{s}}_{t} \|^2]\\
				&\qquad \quad + 2  \alpha^2 v_{\rm min}^{-1}  \beta_1^2 \frac{1 + \rho_A^2}{1 - \rho_A^2} \mathbb{E} [\| \bm{m}_{t} - \tilde{\bm{m}}_{t}\|^2].
			\end{split}
		\end{equation}
		Furthermore, it holds from Lemma~\ref{ratio_sum} that
		\begin{equation} \label{sum_x1}
			\begin{split}
				&\sum_{t=1}^{T}\mathbb{E} [\|\bm{x}_{t} - \tilde{\bm{x}}_{t}\|^2] \leq \frac{2}{1 - \rho_A^2} \Delta_1 \\
				& \qquad \quad + 4 \alpha^2 v_{\rm min}^{-1} (1-\beta_1)^2 \frac{1 + \rho_A^2}{(1 - \rho_A^2)^2} \sum_{t=1}^{T}\mathbb{E} [\| \bm{s}_{t} - \tilde{\bm{s}}_{t} \|^2]\\
				&\qquad \quad + 4  \alpha^2 v_{\rm min}^{-1}  \beta_1^2 \frac{1 + \rho_A^2}{(1 - \rho_A^2)^2} \sum_{t=1}^{T}\mathbb{E} [\| \bm{m}_{t} - \tilde{\bm{m}}_{t}\|^2].
			\end{split}
		\end{equation}
		
		To proceed, we apply Lemma~\ref{lm_mcon} to (\ref{sum_x1}) and obtain
		\begin{equation}
			\begin{split}
				&\sum_{t=1}^{T}\mathbb{E} [\|\bm{x}_{t} - \tilde{\bm{x}}_{t}\|^2] \\
				& \leq \frac{  \alpha^2 v_{\rm min}^{-1} (8 + 32  \beta_1^2 ) }{(1 - \rho_A^2)^2}  \sum_{t=1}^{T} \mathbb{E} [\| \bm{s}_{t} \!-\! \tilde{\bm{s}}_t\|^2] + \frac{2}{1 - \rho_A^2} \Delta_1\\
				& \leq \frac{  40\alpha^2 v_{\rm min}^{-1}}{(1 - \rho_A^2)^2}  \sum_{t=1}^{T} \mathbb{E} [\| \bm{s}_{t} \!-\! \tilde{\bm{s}}_t\|^2] + \frac{2}{1 - \rho_A^2} \Delta_1,
			\end{split}
		\end{equation}
	where the second inequality holds by the relaxation $\beta_1 < 1$.
		This completes the proof.	
	\end{proof}

	\begin{proof}[III: Proof of Lemma~\ref{lm_mbar}] 
		It follows from (\ref{m_update}) and (\ref{bar_def}) that
		\begin{equation}\label{m_bar}
			\begin{split}
				&\|\tilde{\bm{m}}_{t+1} \| = \| \beta_1 \tilde{\bm{m}}_{t} + (1-\beta_1) \tilde{\bm{s}}_{t}\| \\
				&\leq \beta_1 \| \tilde{\bm{m}}_{t}\| + (1-\beta_1)  \sqrt{n} \| \bar{s}_t\| \\
				& = \beta_1 \| \tilde{\bm{m}}_{t}\| + (1-\beta_1) \sqrt{n}  \| \bar {s}_{t} - \nabla f(\bar{x}_t) + \nabla f(\bar{x}_t) \| \\
				& = \beta_1 \| \tilde{\bm{m}}_{t}\| + (1-\beta_1) \sqrt{n} \times \\
				& \quad \bigg\| \frac{1}{n} \sum_{i=1}^{n} \big[\nabla f_i({x}_{t,i}) + \eta_{t,i} - \nabla f_i(\bar{x}_t)\big] + \nabla f(\bar{x}_t) \bigg\| \\
				& \leq \beta_1 \| \tilde{\bm{m}}_{t}\| + (1-\beta_1) \sqrt{n} \bigg[  \frac{1}{n} \sum_{i=1}^{n} \| \nabla f_i({x}_{t,i}) - \nabla f_i(\bar{x}_t)\| \\
				& \qquad \qquad \qquad \qquad \qquad \quad \;\;  +  \bigg\|\frac{1}{n} \sum_{i=1}^{n} \eta_{t,i} \bigg\| + \|\nabla f(\bar{x}_t) \| \bigg] \\
				& \leq 	 \beta_1 \| \tilde{\bm{m}}_{t}\| + L(1-\beta_1)\|\bm{x}_t - \tilde{\bm{x}}_t\| + (1-\beta_1) \sqrt{n}\| \nabla f(\bar{x}_t) \|  \\
				& \quad+  (1-\beta_1)\sqrt{n} \bigg\|\frac{1}{n} \sum_{i=1}^{n} \eta_{t,i} \bigg\|,
			\end{split}
		\end{equation}
		where the first equality holds by Lemma~\ref{lm_basic_relation}(b), and the last inequality follows from Assumption~\ref{ass_smooth}.
		
		From Lemma~\ref{lm_sqr}, let $a = \beta_1 \| \tilde{\bm{m}}_{t}\|$ and $b = L(1-\beta_1) \|\bm{x}_t - \tilde{\bm{x}}_t\| + (1-\beta_1) \sqrt{n}\| \nabla f(\bar{x}_t) \| +  (1-\beta_1)\sqrt{n} \big\|\frac{1}{n} \sum_{i=1}^{n} \eta_{t,i} \big\|$ and $\theta = \frac{1 - \beta_1^2}{2 \beta_1^2} > 0$. Then, (\ref{m_bar}) becomes 
		\begin{equation}
			\begin{split}
				&\mathbb {E} \big[\|\tilde{\bm{m}}_{t+1} \|^2\big] \!\leq\!  \frac{1 + \beta_1^2}{2} \mathbb {E} \big[\| \tilde{\bm{m}}_{t}\|^2\big] + 3(1 \!-\! \beta_1)^2 \frac{1+ \beta_1^2}{1 - \beta_1^2}  \times  \\
				& \quad \;\; \bigg(  L^2 \mathbb {E}\big[\|\bm{x}_t - \tilde{\bm{x}}_t\|^2\big] + {n} \mathbb {E}\big[\| \nabla f(\bar{x}_t) \|^2\big]  + \sum_{i=1}^{n} \mathbb {E} \big[\| \eta_{t,i}\|^2\big] \bigg).
			\end{split}
		\end{equation}
		Finally, by using Lemma~\ref{ratio_sum}, one has for $T \geq 1$ that 
		\begin{equation}
			\begin{split}
				&\sum_{t=1}^{T}\mathbb {E} \big[\|\tilde{\bm{m}}_{t} \|^2\big]  \leq  \frac{2}{1 - \beta_1^2} \|\|\tilde{\bm{m}}_{1}\|^2 + 3 (1 \!-\! \beta_1)^2 \frac{1+ \beta_1^2}{(1 - \beta_1^2)^2} \times  \\
				& \quad \bigg(  L^2 \sum_{t=1}^{T} \mathbb {E}\big[\|\bm{x}_t \!-\! \tilde{\bm{x}}_t\|^2\big] \!+ \!{n} \mathbb {E}\big[\| \nabla f(\bar{x}_t) \|^2\big]  \!+\! \sum_{i=1}^{n} \mathbb {E} \big[\| \eta_{t,i}\|^2\big] \bigg) \\
				& \leq 6 L^2 \sum_{t=1}^{T} \mathbb {E}\big[\|\bm{x}_t - \tilde{\bm{x}}_t\|^2\big]  + 6 {n} \sum_{t=1}^{T} \mathbb {E}\big[\| \nabla f(\bar{x}_t) \|^2\big]  + 6 n \sigma^2 T \\
				& \leq \frac{ 240 L^2 \alpha^2 v_{\rm min}^{-1} }{(1 - \rho_A^2)^2}  \sum_{t=1}^{T} \mathbb{E} [\| \bm{s}_{t} \!-\! \tilde{\bm{s}}_t\|^2]  +  6 {n} \sum_{t=1}^{T} \mathbb {E}\big[\| \nabla f(\bar{x}_t) \|^2\big]  \\
				& \quad+ 6 n \sigma^2 T  + \frac{12 L^2}{1 - \rho_A^2} \Delta_1,
			\end{split}
		\end{equation}	
	where the second inequality holds from the fact that $m_{1,i} = \bm{0}_d$ for all $i$.
	This completes the proof.
	\end{proof}
	
	\begin{proof}[IV: Proof of Lemma~\ref{lm_v_con}] 
		It follows from (\ref{s_con1}) that
		\begin{equation}
			\begin{split}
				&\bm{s}_t = \tilde{\bm{s}}_t + (\bm{s}_t - \tilde{\bm{s}}_t) \\
				&= \bm{1}_n \!\otimes\! \bigg(\frac{1}{n}\sum_{i=1}^{n}g_{t,i}\bigg) + \bigg[\bigg(A - \frac{\bm{1}_n \bm{1}_n'}{n}\bigg) \!\otimes\! I_d \bigg] (\bm{s}_{t-1} - \tilde{\bm{s}}_{t-1}) \\
				&\quad + \bigg[\Big( I_n - \frac{\bm{1}_n\bm{1}_n'}{n} \Big) \otimes I_d \bigg](\bm{g}_{t+1} - \bm{g}_t) \\
				& = \dots \\
				& = \bm{1}_n \!\otimes\! \bigg(\frac{1}{n}\sum_{i=1}^{n}g_{t,i}\bigg) \!+\! \sum_{\tau = 1}^{t} \bigg[\bigg(A \!-\! \frac{\bm{1}_n \bm{1}_n'}{n}\bigg)^\tau \!\otimes\! I_d \bigg] (\bm{g}_{t+1} \!-\! \bm{g}_t) \\
				&\quad + \bigg[\bigg( I_n - \frac{\bm{1}_n\bm{1}_n'}{n} \bigg) \otimes I_d \bigg](\bm{g}_{t+1} - \bm{g}_t).
			\end{split}
		\end{equation}
		Then, according to Assumption~\ref{assumption_sto_grad}(b), it holds for all $l = 1, \dots, nd$ that
		\begin{equation} \label{s_uni_bound}
			\begin{split}
				&\big| [\bm{s}_t]_l \big| \leq \big| [\tilde{\bm{s}}_t]_l \big| + \big| [\bm{s}_t - \tilde{\bm{s}}_t]_l \big| \\
				&\leq \bigg| \bigg[\bm{1}_n\otimes \bigg(\frac{1}{n}\sum_{i=1}^{n}g_{t,i} \bigg) \bigg]_l \bigg| + \sum_{\tau=0}^{t} \rho_A^\tau \big| [ \bm{g}_{t+1} - \bm{g}_t]_l \big|\\
				& \leq G + \frac{2G}{1-\rho_A} = \frac{G(3-\rho_A)}{1-\rho_A}.
			\end{split}
		\end{equation}
		
		By (\ref{v_update}), (\ref{v3_update}), (\ref{bar_def}) and the non-extension property of the clipping operator, one has
		\begin{equation} \label{v_error}
			\begin{split}
				&\| \bm{v}_{t+1} - \tilde{\bm{v}}_{t+1}\| \\
				& \leq \bigg\| \beta_2\bm{v}_{t} + (1 - \beta_2) \bm{s}_t \odot \bm{s}_t  - \beta_2 \tilde{\bm{v}}_{t} \\
				& \quad \;\;- (1 - \beta_2) \bm{1}_n \otimes \bigg(\frac{1}{n}\sum_{j=1}^{n} s_{t,i} \odot {s}_{t,i} \bigg)  \bigg\| \\
				& \leq \beta_2 \| \bm{v}_{t} - \tilde{\bm{v}}_{t}\| + (1 - \beta_2) \bigg\| \bigg[\bigg(I_n - \frac{\bm{1}_n \bm{1}_n'}{n}\bigg) \otimes I_d\bigg] \big(\bm{s}_t \odot \bm{s}_t \big) \bigg\| \\
				& \leq \beta_2 \| \bm{v}_{t} \!-\! \tilde{\bm{v}}_{t}\| \!+\! \frac{G (3 \!-\! \rho_A) (1 \!-\! \beta_2)}{1-\rho_A} \bigg\| \bigg[\bigg(I_n \!- \frac{\bm{1}_n \bm{1}_n'}{n}\bigg) \otimes I_d\bigg] \bm{s}_t  \bigg\| \\
				& \leq \beta_2\| \bm{v}_{t} - \tilde{\bm{v}}_{t}\| + \frac{G (3-\rho_A) (1-\beta_2)}{1-\rho_A} \| \bm{s}_t - \tilde{\bm{s}}_t \|,
			\end{split}
		\end{equation}
		where the third inequality holds by (\ref{s_uni_bound}) and the property of Hadamard product. 
		
		Then, by using Lemma~\ref{lm_sqr} with $\theta = \frac{1-\beta_2^2}{2 \beta_2^2}$ and taking the total expectation, one has
		\begin{equation} \label{v_t_relation}
			\begin{split}
				&\mathbb{E} [\| \bm{v}_{t+1} - \tilde{\bm{v}}_{t+1}\|^2]  \leq \frac{1 + \beta_2^2}{2} \mathbb{E} [\| \bm{v}_{t} - \tilde{\bm{v}}_{t}\|^2] \\
				&\quad + \Big(\frac{G (3-\rho_A) (1-\beta_2)}{(1-\rho_A)}\Big)^2 \frac{(1 + \beta_2^2)}{1 - \beta_2^2} \mathbb{E} [\| \bm{s}_t - \tilde{\bm{s}}_t \|^2]\\
				& \leq\frac{1 + \beta_2^2}{2} \mathbb{E} [\| \bm{v}_{t} - \tilde{\bm{v}}_{t}\|^2] + \frac{18G^2 (1 - \beta_2)^2 }{(1-\rho_A)^2 (1 - \beta_2^2)} \mathbb{E} [\| \bm{s}_t - \tilde{\bm{s}}_t \|^2].
			\end{split}
		\end{equation}
		
		Furthermore, we apply Lemma~\ref{ratio_sum} to (\ref{v_t_relation}) and obtain
		\begin{equation} 
			\begin{split}
				&\sum_{t=1}^{T} \mathbb{E} [\| \bm{v}_{t} - \tilde{\bm{v}}_{t}\|^2]  \leq \frac{2}{1 - \beta_2^2} \Delta_2 \\
				&\quad + \frac{36G^2 (1 - \beta_2)^2 }{(1-\rho_A)^2 (1 - \beta_2^2)^2} \sum_{t=1}^{T}  \mathbb{E} [\| \bm{s}_t - \tilde{\bm{s}}_t \|^2] \\
				& \leq \frac{2}{1 - \beta_2^2} \Delta_2  + \frac{36G^2 }{(1-\rho_A)^2} \sum_{t=1}^{T}  \mathbb{E} [\| \bm{s}_t - \tilde{\bm{s}}_t \|^2].
			\end{split}
		\end{equation}
		This completes the proof.
	\end{proof}
	
	\begin{proof}[V: Proof of Lemma~\ref{lm_s_con}] 
		It follows from (\ref{s_update}) and (\ref{bar_def}) that
		\begin{equation} \label{s_con1}
			\begin{split}
				&\|\bm{s}_{t+1} - \tilde{\bm{s}}_{t+1}\| \\
				&= \bigg\| (A \!\otimes\! I_d ) \bm{s}_t \!+\! \bm{g}_{t+1} \!-\! \bm{g}_t
				-  \Big(\frac{\bm{1}_n\bm{1}_n'}{n} \!\otimes\! I_d \Big)(\bm{s}_t \!+\! \bm{g}_{t+1} \!-\! \bm{g}_t) \bigg\| \\
				&= \bigg\| \bigg[ \Big(A \!-\! \frac{\bm{1}_n\bm{1}_n'}{n} \Big) \!\otimes\! I_d \bigg] \bm{s}_t \!+\! \bigg[ \Big(I_n \!-\! \frac{\bm{1}_n\bm{1}_n'}{n} \Big) \!\otimes\! I_d \bigg](\bm{g}_{t+1} - \bm{g}_t) \bigg\| \\
				& \leq \bigg\| \Big[ \Big(A - \frac{\bm{1}_n\bm{1}_n'}{n} \Big) \otimes I_d \Big] (\bm{s}_t - \tilde{\bm{s}}_t) \bigg\| \\
				& \quad + \bigg\| \Big[ \Big(I_n - \frac{\bm{1}_n\bm{1}_n'}{n} \Big) \otimes I_d \Big] (\bm{g}_{t+1} - \bm{g}_t) \bigg\| \\
				& \leq \rho_A \| \bm{s}_t - \tilde{\bm{s}}_t\| + \|\bm{g}_{t+1} - \bm{g}_t \| \\
				& \leq \rho_A \| \bm{s}_t - \tilde{\bm{s}}_t\| + L\|\bm{x}_{t+1} - \bm{x}_t \| +  \| \boldsymbol{\eta}_{t+1} \| + \| \boldsymbol{\eta}_{t} \|\\
				&= \rho_A \| \bm{s}_t - \tilde{\bm{s}}_t\| + L \big\| (A \otimes I_d) \bm{x}_{t} - \alpha V_{t+1}^{-1/2} \bm{m}_{t+1}	 - \bm{x}_t \big\| \\
				& \quad + \| \boldsymbol{\eta}_{t+1} \| + \|\boldsymbol{\eta}_{t} \| \\
				&\leq \rho_A \| \bm{s}_t - \tilde{\bm{s}}_t\| + L\|[(A-I_n) \otimes I_d](\bm{x}_{t} - \tilde{\bm{x}}_t) \| \\
				& \quad + \alpha v_{\rm min}^{-1/2} \| \bm{m}_{t+1} \| + \| \boldsymbol{\eta}_{t+1} \| + \|\boldsymbol{\eta}_{t} \|\\
				&\leq \rho_A \| \bm{s}_t - \tilde{\bm{s}}_t\| + L \rho(A-I_n) \|\bm{x}_{t} - \tilde{\bm{x}}_t \| \\
				& \quad + \alpha v_{\rm min}^{-1/2} (\| \bm{m}_{t+1} \!-\! \tilde{\bm{m}}_{t+1} \| \!+\! \| \tilde{\bm{m}}_{t+1} \|) \!+\! \| \boldsymbol{\eta}_{t+1} \| \!+\! \|\boldsymbol{\eta}_{t} \| \\
				&\leq \rho_A \| \bm{s}_t - \tilde{\bm{s}}_t\| + 2L \|\bm{x}_{t} - \tilde{\bm{x}}_t \| \\
				& \quad + \alpha v_{\rm min}^{-1/2} (\| \bm{m}_{t+1} \!-\! \tilde{\bm{m}}_{t+1} \| \!+\! \| \tilde{\bm{m}}_{t+1} \|) \!+\! \| \boldsymbol{\eta}_{t+1} \| \!+\! \|\boldsymbol{\eta}_{t} \|,
			\end{split}
		\end{equation}
		where the first, the second, the fourth and the last inequality holds by the facts $\big[\big(A - \frac{\bm{1}_n\bm{1}_n'}{n} \big) \otimes I_d \big] \tilde{\bm s}_t = \bm{0}_{nd}$, $\rho \big(I_n - \frac{\bm{1}_n\bm{1}_n'}{n} \big) = 1$, $\big[\big(A - I_{n} \big) \otimes I_d \big] \tilde{\bm x}_t = \bm{0}_{nd}$ and	 $\rho (A - I_n) \leq 2$ (by Lemma~\ref{rho_bound}), respectively.
		
		From (\ref{m_con}), (\ref{m_bar}), (\ref{s_con1}) and Lemma~\ref{lm_sqr}, it holds for $\theta_4>0$ that
		\begin{equation}\label{s_con2}
			\begin{split}
				&\|\bm{s}_{t+1} - \tilde{\bm{s}}_{t+1}\|^2 \\
				&\leq \bigg[ \rho_A \| \bm{s}_t \!-\! \tilde{\bm{s}}_t\| \!+ \! 2L \|\bm{x}_{t} \!-\! \tilde{\bm{x}}_t \| \!+ \!\| \boldsymbol{\eta}_{t+1} \| \!+\! \|\boldsymbol{\eta}_{t}\| \\
				& \quad  +\! \alpha v_{\rm min}^{-1/2} \bigg[  \beta_1\| \bm{m}_{t} \!-\! \tilde{\bm{m}}_{t} \| \!+\! (1 \!-\! \beta_1) \| \bm{s}_{t} \!-\! \tilde{\bm{s}}_{t} \| \!+\! \beta_1\| \tilde{\bm{m}}_{t} \| \\
				& \quad  +\! (1 \!-\! \beta_1) \bigg( \! L\| \bm{x}_{t} \!-\! \tilde{\bm{x}}_{t} \| \!+\! \sqrt{n} \| \nabla f(\bar{x}_t)\| \!+\!  \sqrt{n} \bigg\|\frac{1}{n} \sum_{i=1}^{n}\eta_{t,i} \bigg\|  \bigg) \! \bigg] \bigg]^2\\
				&\leq \rho_A^2 (1+\theta_4)  \| \bm{s}_t - \tilde{\bm{s}}_t\|^2 +9(1+1/\theta_4) \times \\
				& \quad \bigg[ 4L^2 \|\bm{x}_{t} - \tilde{\bm{x}}_t \|^2 + \sum_{i=1}^{n}(\| {\eta}_{t+1,i} \|^2 + \| {\eta}_{t,i} \|^2)  \\
				&\quad + \! \alpha^2 v_{\rm min}^{-1} \bigg[\beta_1^2\| \bm{m}_{t} \!-\! \tilde{\bm{m}}_{t} \|^2 \!+\! (1 \!-\! \beta_1)^2 \| \bm{s}_{t} \!-\! \tilde{\bm{s}}_{t} \|^2  \!+\! \beta_1^2\| \tilde{\bm{m}}_{t} \|^2\\
				& \quad  + (1 \!-\! \beta_1)^2 \bigg( {L^2} \| \bm{x}_{t} \!-\! \tilde{\bm{x}}_{t} \|^2 \!+\! n \| \nabla f(\bar{x}_t)\|^2 \!+\! \sum_{i=1}^{n} \|\eta_{t,i}\|^2 \bigg) \bigg] \bigg].
			\end{split}
		\end{equation}
		Since $\beta_1 \in (0,1)$, by setting $\theta = \frac{1 - \rho_A^2}{2 \rho_A^2} $ and taking the total expectation of~(\ref{s_con2}), one has
		\begin{equation}\label{s_t_relation}
			\begin{split}
				&\mathbb{E}\big[\|\bm{s}_{t+1} - \tilde{\bm{s}}_{t+1}\|^2\big] \\
				&\leq \Big( \frac{1 + \rho_A^2}{2} \!+\! 9\alpha^2 v_{\rm min}^{-1}(1-\beta_1)^2 \frac{1 + \rho_A^2}{1 - \rho_A^2} \Big) \mathbb{E} \big[\| \bm{s}_t \!-\! \tilde{\bm{s}}_t\|^2\big] \\
				&\quad + 9 L^2\frac{1 + \rho_A^2}{1 - \rho_A^2} \Big( 4 + \frac{\alpha^2 v_{\rm min}^{-1}}{n}(1 - \beta_1)^2  \Big) \times \\
				& \quad\quad \mathbb{E}\big[\|\bm{x}_{t} - \tilde{\bm{x}}_t \|^2\big] + 9 \alpha^2 v_{\rm min}^{-1} \frac{1 + \rho_A^2}{1 - \rho_A^2}   \Big( \beta_1^2 \mathbb{E} \big[\| \bm{m}_{t} \!-\! \tilde{\bm{m}}_{t} \|^2\big]   \\
				& \quad+ \beta_1^2 \mathbb{E}\big[\| \tilde{\bm{m}}_{t} \|^2\big]+ n (1 \!-\! \beta_1)^2  \mathbb{E} \big[\| \nabla f(\bar{x}_t)\|^2 \big] \Big) \\
				& \quad+ \frac{18}{1 - \rho_A^2} ( 3 + \alpha^2 v_{\rm min}^{-1} (1-\beta_1)^2)n \sigma^2 \\
				&\leq \Big( \frac{3 + \rho_A^2}{4}  \Big) \mathbb{E} \big[\| \bm{s}_t \!-\! \tilde{\bm{s}}_t\|^2\big] + \frac{90 L^2}{1 - \rho_A^2} \mathbb{E}\big[\|\bm{x}_{t} - \tilde{\bm{x}}_t \|^2\big] \\
				& \quad +  \frac{18\alpha^2 v_{\rm min}^{-1} \beta_1^2}{1 - \rho_A^2}   \Big( \mathbb{E} \big[\| \bm{m}_{t} \!-\! \tilde{\bm{m}}_{t} \|^2\big] + \mathbb{E}\big[\| \tilde{\bm{m}}_{t} \|^2\big] \Big) \\
				& \quad+ \frac{18\alpha^2 v_{\rm min}^{-1} }{1 - \rho_A^2}  n  \mathbb{E} \big[\| \nabla f(\bar{x}_t)\|^2 \big]  + \frac{54}{1 - \rho_A^2} n \sigma^2 ,
			\end{split}
		\end{equation}
		where the second inequality holds since $\alpha^2 \leq \frac{v_{\min} (1-\rho_A^2)^2}{72}$ and $\rho_A \in [0,1)$.
		
		By applying Lemma~\ref{ratio_sum} to~(\ref{s_t_relation}), we obtain
		\begin{equation}\label{s_sum}
			\begin{split}
				&\sum_{t=1}^{T} \mathbb{E}\big[\|\bm{s}_{t} - \tilde{\bm{s}}_{t}\|^2\big] \\
				&\leq \Big( \frac{4}{1 - \rho_A^2}  \Big) \Delta_3 + \frac{360 L^2}{(1 - \rho_A^2)^2} \sum_{t=1}^{T} \mathbb{E}\big[\|\bm{x}_{t} - \tilde{\bm{x}}_t \|^2\big] \\
				& \quad +  \frac{72 \alpha^2 v_{\rm min}^{-1} \beta_1^2}{(1 - \rho_A^2)^2}   \Big( \sum_{t=1}^{T} \mathbb{E} \big[\| \bm{m}_{t} \!-\! \tilde{\bm{m}}_{t} \|^2\big] + \sum_{t=1}^{T} \mathbb{E}\big[\| \tilde{\bm{m}}_{t} \|^2\big] \Big) \\
				& \quad+ \frac{72 \alpha^2 v_{\rm min}^{-1} }{(1 - \rho_A^2)^2}  n  \sum_{t=1}^{T}\mathbb{E} \big[\| \nabla f(\bar{x}_t)\|^2 \big]  + \frac{216}{(1 - \rho_A^2)^2} n T \sigma^2  \\
				&\leq \Big( \frac{4}{1 - \rho_A^2}  \Big) \Delta_3 + \frac{360 L^2}{(1 - \rho_A^2)^2} \sum_{t=1}^{T} \mathbb{E}\big[\|\bm{x}_{t} - \tilde{\bm{x}}_t \|^2\big] \\
				& \quad +  \frac{72 \alpha^2 v_{\rm min}^{-1} \beta_1^2}{(1 - \rho_A^2)^2}   \Big( \sum_{t=1}^{T} \mathbb{E} \big[\| \bm{m}_{t} \!-\! \tilde{\bm{m}}_{t} \|^2\big] + \sum_{t=1}^{T} \mathbb{E}\big[\| \tilde{\bm{m}}_{t} \|^2\big] \Big) \\
				& \quad+ \frac{72 \alpha^2 v_{\rm min}^{-1} }{(1 - \rho_A^2)^2}  n  \sum_{t=1}^{T}\mathbb{E} \big[\| \nabla f(\bar{x}_t)\|^2 \big]  + \frac{216}{(1 - \rho_A^2)^2} n T \sigma^2  \\
				&\leq N_1 \alpha^2\sum_{t=1}^{T} \mathbb{E}\big[\|\bm{s}_{t} - \tilde{\bm{s}}_{t}\|^2\big] +  N_2 n  \sum_{t=1}^{T}\mathbb{E} \big[\| \nabla f(\bar{x}_t)\|^2 \big]  \\
				& \quad + N_3 nT \sigma^2  + N_4 \Delta
			\end{split}
		\end{equation}		
		with $N_1$-$N_4$ given by (\ref{def_N1234}), where the second inequality holds by applying Lemmas~\ref{lm_mcon}-\ref{lm_v_con}. If we set $\alpha^2 \leq \frac{1}{2 N_1}$, then it is easy to prove that
		\begin{equation}
			\begin{split}
				& (1 - N_1 \alpha^2)\sum_{t=1}^{T} \mathbb{E}\big[\|\bm{s}_{t} - \tilde{\bm{s}}_{t}\|^2\big] \leq \frac{1}{2}\sum_{t=1}^{T} \mathbb{E}\big[\|\bm{s}_{t} - \tilde{\bm{s}}_{t}\|^2\big] \\
				& \leq n N_2  \sum_{t=1}^{T}\mathbb{E} \big[\| \nabla f(\bar{x}_t)\|^2 \big]   + nT N_3  \sigma^2  + N_4 \Delta,
			\end{split}
		\end{equation}
		which completes the proof.	
	\end{proof}
	
	\section*{Appendix D: Proof of Corollary~\ref{Col_1}}
	
	\begin{proof}
		In Theorem~\ref{The_1}, we have already bounded the first part as follows:
		\begin{equation} \label{inequality_col_1}
			\begin{split}
				& \frac{1}{T}\sum_{t=1}^{T} \mathbb{E} [\|\nabla f(\bar{x}_t)\|^2] \\
				& \quad \quad \leq \frac{f(\bar{x}_{1}) - f^* + N_4' \Delta}{T \alpha (v_{\max}^{-1/2} - \alpha N_2')} + \frac{ N_3'  \sigma^2}{\alpha (v_{\max}^{-1/2} - \alpha N_2')}.
			\end{split}
		\end{equation}
	
	Then, from (\ref{inequality_col_1}) and Lemma~\ref{lm_s_con}, one has 
	\begin{equation} \label{inequality_col_2}
		\begin{split}
			&\frac{1}{n T} \sum_{t=1}^{T} \mathbb{E}\big[\|\bm{s}_{t} - \tilde{\bm{s}}_{t}\|^2\big]  \leq \frac{ 2 N_2 (f(\bar{x}_{1}) - f^*) }{T \alpha (v_{\max}^{-1/2} - \alpha N_2')} \\
			& + \left(\frac{2 N_2 N_4'}{\alpha (v_{\max}^{-1/2} - \alpha N_2')} + \frac{2 N_4}{n} \right) \frac{\Delta}{T} \\
			& + \left(\frac{ 2 N_2 N_3' }{\alpha (v_{\max}^{-1/2} - \alpha N_2')} + 2 N_3 \right) \sigma^2 .
		\end{split}
	\end{equation}

To proceed, we substitute (\ref{inequality_col_2}) into (\ref{inequality_lm_x}) and get 
\begin{equation}
	\begin{split}
		&\frac{1}{nT} \sum_{t=1}^{T} \sum_{i=1}^{n} \mathbb{E} \left[\|x_{t,i} - \bar{x}_t \|^2\right] = \frac{1}{nT} \sum_{t=1}^{T}\mathbb{E} [\|\bm{x}_{t} - \tilde{\bm{x}}_{t}\|^2] \\
		& \leq \frac{ 40 \alpha^2 v_{\rm min}^{-1} }{(1 - \rho_A^2)^2} \cdot \frac{1}{n T} \sum_{t=1}^{T} \mathbb{E} [\| \bm{s}_{t} - \tilde{\bm{s}}_t\|^2] + \frac{2}{n (1 - \rho_A^2)} \cdot \frac{\Delta_1}{T} \\
		& \leq \frac{ 80 N_2  \alpha^2 v_{\rm min}^{-1} }{\alpha (1 - \rho_A^2)^2 (v_{\max}^{-1/2} - \alpha N_2')} \cdot \frac{f(\bar{x}_{1}) - f^*}{T} \\
		& \quad +  \left[\frac{ 80 \alpha^2 v_{\rm min}^{-1} }{(1 - \rho_A^2)^2} \! \left(\frac{N_2 N_4'}{\alpha (v_{\max}^{-1/2} \!-\! \alpha N_2')} \!+\! \frac{N_4}{n} \right) \!+\!  \frac{2}{n (1 \!-\! \rho_A^2)}\right] \cdot \frac{\Delta}{T} \\
		& \quad + \frac{ 40 \alpha^2 v_{\rm min}^{-1} }{(1 - \rho_A^2)^2} \cdot  \left(\frac{ 2 N_2 N_3' }{\alpha (v_{\max}^{-1/2} - \alpha N_2')} + 2 N_3 \right) \sigma^2,
	\end{split}
\end{equation}
which completes the proof.
	\end{proof}

\end{document}